\documentclass[11pt]{article}
\usepackage[english]{babel}
\usepackage[latin1]{inputenc}
\usepackage[T1]{fontenc}
\usepackage{graphics,graphicx} 

\usepackage{latexsym}
\usepackage{amsfonts}
\usepackage{amsthm}
\usepackage{amsmath}
\usepackage{eepic,epic}
\usepackage{graphicx}
\usepackage[hmargin=2.3cm,vmargin=2.5cm]{geometry}

\usepackage{numprint}

\newcommand\p{\circle*{0.2}}

\newtheorem{theorem}{Theorem}

\newtheorem{lemma}{Lemma}

\title{On word-representability of polyomino triangulations}
\author{Prosper Akrobotu\thanks{AIMS-Ghana, P. O. BOX DL 676, Adisadel - Cape Coast. \newline Email: prosper@aims.edu.gh}, Sergey Kitaev\thanks{School of Computer and Information Sciences, University of Strathclyde, Glasgow, G1 1HX, UK. \newline Email: sergey.kitaev@cis.strath.ac.uk}, and Zuzana Mas\'{a}rov\'{a}\thanks{Department of Combinatorics and Optimization, University of Waterloo, Waterloo, N2L 3G1, Canada. \newline Email: zmasarova@uwaterloo.ca}}

\begin{document}

\maketitle
\thispagestyle{empty}

\begin{abstract}  A graph $G=(V,E)$ is word-representable if there exists a word $w$ over the alphabet $V$ such that letters $x$ and $y$ alternate in $w$ if and only if $(x,y)$ is an edge in $E$. Some graphs are word-representable, others are not. It is known that a graph is word-representable if and only if it accepts a so-called semi-transitive orientation. 

The main result of this paper is showing that a triangulation of any convex polyomino is word-representable if and only if it is 3-colorable. We demonstrate that this statement is not true for an arbitrary polyomino. We also show that the graph obtained by replacing each $4$-cycle in a polyomino by the complete graph $K_4$ is word-representable. We employ semi-transitive orientations to obtain our results. \\

 \end{abstract}

\section{Introduction}\label{intro}

A graph $G=(V,E)$ is word-representable if there exists a word $w$ over the alphabet $V=V(G)$ such that letters $x$ and $y$ alternate in $w$ if and only if $(x,y)$ is an edge in $E=E(G)$. For example, the cycle graph on 4 vertices labeled by 1, 2, 3 and 4 in clockwise direction can be represented by the word 14213243.  There is a long line of research on word-representable graphs \cite{CKL}--\cite{KitSalSevUlf2}, and the current paper is a continuation of this line of research. A comprehensive introduction to the theory of word-representable graphs will be given in~\cite{KL}.

A directed graph (digraph) $G=(V,E)$ is \emph{semi-transitive} if it has no directed cycles
and for any directed path $v_1v_2\cdots v_k$ with $k\geq 4$ and $v_i\in V$, 
either $v_1v_k\not\in E$ or $v_iv_j\in E$ for all $1\le i<j\le k$. 
In the second case, when $v_1v_k\in E$, we say that $v_1v_k$ is a {\em shortcut}. 
The importance of this notion is due to the following result proved in \cite{HKP2011}.

\begin{theorem}[\cite{HKP2011}]\label{thm:semi-trans}  
A graph is word-representable if and only if it admits a semi-transitive orientation. 
\end{theorem}

A graph is $k$-colorable if its vertices can be colored in at most $k$ colors so that no pair of vertices having the same color is connected by an edge. A direct corollary to the last theorem is the following result also relevant to our paper.

\begin{theorem}[\cite{HKP2011}]\label{thm:3col} 
All $3$-colorable graphs are word-representable.
\end{theorem}

\begin{proof} 
Partitioning a 3-colorable graph in three independent sets, say I, II and III, 
and orienting all edges in the graph so that they are oriented from I to II and III, and from II to III, 
we obtain a semi-transitive orientation.
\end{proof}

We note that, for $k\geq 4$, there are examples of non-word-representable graphs that are $k$-colorable, but not 3-colorable. For example, the wheel $W_5$ on 6 vertices is such a graph.

A {\em polyomino} is a plane geometric figure formed by joining one or more equal squares edge to edge. Letting conners of squares in a polyomino be vertices, we can treat polyominos as graphs. In particular, well-known {\em grid graphs} are obtained from polyominos in this way. A particular class of graphs of our interest is given by {\em convex polyominos}. A polyomino is said to be {\em column convex} if its intersection with any vertical line is convex (in other words, each column has no holes). Similarly, a polyomino is said to be {\em row convex} if its intersection with any horizontal line is convex. A polyomino is said to be convex if it is row and column convex.

We will consider {\em triangulations} of a polyomino. Note that no triangulation is 2-colorable -- at least three colors are needed to color properly a triangulation, while four colors is always enough to colour any triangulation since we deal with planar graphs and it is well-known that such graphs are 4-colorable. Not all triangulations of a polyomino are 3-colorable -- for example, see Figure~\ref{non-repr-triang} for non-3-colorable triangulations (which are straightforward to check to require four colors, and also to be the only such triangulations, up to rotations, of a $3 \times 3$ grid graph). The main result of this paper is the following theorem.

\begin{theorem}\label{main-thm} A triangulation of a convex polyomino is word-representable if and only if it is $3$-colorable. \end{theorem}

In Section~\ref{sec2}, we employ semi-transitive orientations to show that two graphs in Figure~\ref{non-repr-triang} are non-word-representable. These graphs are to be used in the proof of Theorem~\ref{main-thm} in Section~\ref{sec3}. Subsection~\ref{triangArbitrary} shows that Theorem~\ref{main-thm} is not true for an arbitrary polyomino. In Section~\ref{sec4}, we consider a relevant direction of research: we prove that replacing each 4-cycle of a polyomino $\mathcal{P}$ by the complete graph $K_4$ gives a graph $\mathcal{P}_{K_4}$, which is word-representable. 
Finally, in Section~\ref{final-remarks-sec} we provide concluding remarks.

\section{Non-word-representability of graphs in Figure~\ref{non-repr-triang}}\label{sec2}

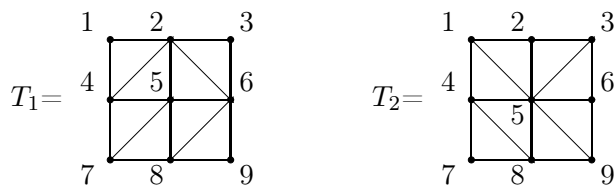
\begin{figure}[h]
\begin{center}
\begin{picture}(6,4.5)

\put(-5,0){


\put(-3.3,1.8){$T_1$=}

\put(-1,4.3){1} \put(1.3,4.3){2} \put(4.3,4.3){3} 
\put(-1,2.3){4} \put(1.3,2.3){5} \put(4.3,2.3){6} 
\put(-1,-0.8){7} \put(1.3,-0.8){8} \put(4.3,-0.8){9} 

\put(0,4){\p} \put(2,4){\p} \put(4,4){\p} 
\put(0,2){\p} \put(2,2){\p} \put(4,2){\p} 
\put(0,0){\p} \put(2,0){\p} \put(4,0){\p} 

\put(0,0){\line(1,0){4}}
\put(0,2){\line(1,0){4}}
\put(0,4){\line(1,0){4}}
\put(0,0){\line(0,1){4}}
\put(2,0){\line(0,1){4}}
\put(4,0){\line(0,1){4}}

\put(0,0){\line(1,1){2.1}}
\put(0,2){\line(1,1){2.1}}
\put(2,0){\line(1,1){2.1}}
\put(2,4){\line(1,-1){2.1}}

}

\put(7,0){


\put(-3.3,1.8){$T_2$=}

\put(-1,4.3){1} \put(1.3,4.3){2} \put(4.3,4.3){3} 
\put(-1,2.3){4} \put(1.3,1.2){5} \put(4.3,2.3){6} 
\put(-1,-0.8){7} \put(1.3,-0.8){8} \put(4.3,-0.8){9} 

\put(0,4){\p} \put(2,4){\p} \put(4,4){\p} 
\put(0,2){\p} \put(2,2){\p} \put(4,2){\p} 
\put(0,0){\p} \put(2,0){\p} \put(4,0){\p} 

\put(0,0){\line(1,0){4}}
\put(0,2){\line(1,0){4}}
\put(0,4){\line(1,0){4}}
\put(0,0){\line(0,1){4}}
\put(2,0){\line(0,1){4}}
\put(4,0){\line(0,1){4}}

\put(0,2){\line(1,-1){2.1}}
\put(0,4){\line(1,-1){4}}
\put(2,2){\line(1,1){2.1}}

}

\end{picture}
\caption{Graphs $T_1$ and $T_2$.} \label{non-repr-triang}
\end{center}
\end{figure}

\begin{theorem}\label{T1-T2} Graphs $T_1$ and $T_2$ in Figure~\ref{non-repr-triang} are not word-representable. \end{theorem}

\begin{proof} 

We show that any attempt to orient edges of the graph $T_{i}$, for
$i=1,2$, necessarily results in creating a shortcut. Hence, $T_{i}$
does not admit a semi-transitive orientation and, by Theorem \ref{thm:semi-trans},
is not word-representable.

Notice that for any of the partial orientations of the 3- or 4-cycle
given in Figure \ref{rules}, there is a unique way of completing
these orientations, also shown in  Figure~\ref{rules}, so that oriented cycles and shortcuts are avoided. This stays true in the context of $T_{i}$,
where, although each 4-cycle is triangulated, it is never a part
of a $K_{4}$, which does admit an alternative semi-transitive
(in fact, transitive) orientation completion.

\begin{figure}[h]
\begin{center}
\begin{picture}(6,4.5)

\put(-17,0){


%

\put(0,0){\p} 
\put(4,0){\p} 
\put(4,3){\p} 

\put(0,0){\vector(1,0){4}}
\put(4,0){\vector(0,1){3}}
\put(0,0){\line(4,3){4}}
\put(5,1){$\longmapsto$}

}
\put(-10,0){

\put(0,0){\p} 
\put(4,0){\p} 
\put(4,3){\p} 

\put(0,0){\vector(1,0){4}}
\put(4,0){\vector(0,1){3}}
\put(0,0){\vector(4,3){4}}

}

\put(-2,0){

\put(4,0){\p}
\put(4,3){\p}
\put(0,3){\p}
\put(0,0){\p}

\put(0,0){\vector(1,0){4}}
\put(4,0){\vector(0,1){3}}
\put(0,0){\line(0,1){3}}
\put(0,3){\line(1,0){4}}

\put(5,1){$\longmapsto$}
}

\put(5,0){


\put(4,0){\p}
\put(4,3){\p}
\put(0,3){\p}
\put(0,0){\p}

\put(0,0){\vector(1,0){4}}
\put(4,0){\vector(0,1){3}}
\put(0,0){\vector(0,1){3}}
\put(0,3){\vector(1,0){4}}

}

\put(12,0){

\put(4,0){\p}
\put(4,3){\p}
\put(0,3){\p}
\put(0,0){\p}

\put(0,0){\vector(1,0){4}}
\put(4,0){\line(0,1){3}}
\put(0,0){\line(0,1){3}}
\put(4,3){\vector(-1,0){4}}

\put(5,1){$\longmapsto$}
}

\put(19.5,0){

\put(4,0){\p}
\put(4,3){\p}
\put(0,3){\p}
\put(0,0){\p}

\put(0,0){\vector(1,0){4}}
\put(4,3){\vector(0,-1){3}}
\put(0,0){\vector(0,1){3}}
\put(4,3){\vector(-1,0){4}}

}

\end{picture}


\caption{Unique way of completing partial orientations on certain subgraphs
of $T_{i}$.}
\label{rules}
\end{center}
\end{figure}
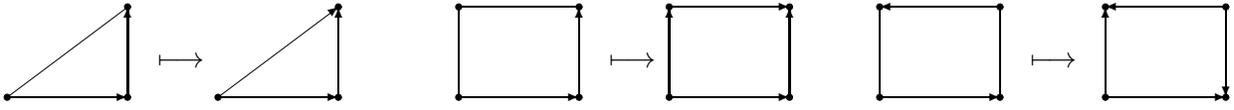

We use the following terminology. \emph{Complete XYW(Z)} refers to
completing the orientations on cycle $XYW(Z)$ according to the respective
cases in Figure \ref{rules}. Instances in which it is not possible
to uniquely determine orientations of any additional edges in $T_{i}$
are referred to as \emph{Branching XY}. Here, one picks a new, still non-oriented edge
$(X,Y)$ of $T_{i}$ and assigns the orientation $X\rightarrow Y$, while, at the
same time, one makes a copy of $T_{i}$ with its partial orientations
and assigns orientation $Y\rightarrow X$ to the edge $(X,Y)$. The new copy is named and examined
later on. Our terminology and relevant abbreviations are
summarized in Table~\ref{table1}.

\begin{table}[h!]  
\centering 
\begin{tabular}{|c|l|} 
\hline 
{\bf Abbreviation} & {\bf Operation} \\ 
\hline 
B & Branch \\
NPOC & Obtain a new partially oriented copy of $T_{i}$ \\
C & Complete \\ 
MC & Move to a copy \\
S & Obtain a shortcut \\

\hline
\end{tabular} 
\caption{List of used operations and their abbreviations.} 
\label{table1}
\end{table}

We now exhaustively search for possible semi-transitive orientations of $T_{i}$. Without loss of generality, the first orientation of an edge in $T_{i}$
can be picked at random (this is because whenever an oriented graph
contains a shortcut, then so does the graph in which all orientations
are reversed). The following two procedures prove
that any orientation of $T_{1}$ or $T_{2}$ necessarily results in
a shortcut. \\

\noindent
{\bf Orienting $T_{1}$. } Name $A$ the first copy of $T_{1}$ with single edge orientation
$78$ and carry out the following operations. 

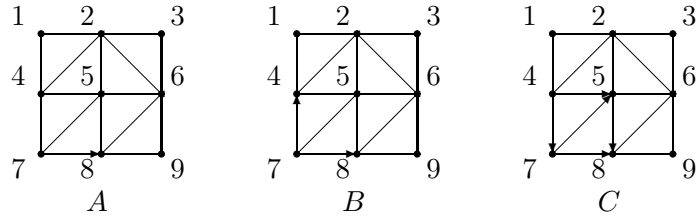
\begin{figure}[h]

\begin{center}
\begin{picture}(6,5)

\put(-7,0){


\put(-1,4.3){1} \put(1.3,4.3){2} \put(4.3,4.3){3} 
\put(-1,2.3){4} \put(1.3,2.3){5} \put(4.3,2.3){6} 
\put(-1,-0.8){7} \put(1.3,-0.8){8} \put(4.3,-0.8){9} 

\put(0,4){\p} \put(2,4){\p} \put(4,4){\p} 
\put(0,2){\p} \put(2,2){\p} \put(4,2){\p} 
\put(0,0){\p} \put(2,0){\p} \put(4,0){\p} 

\put(0,0){\vector(1,0){2}}
\put(0,0){\line(1,0){4}}
\put(0,2){\line(1,0){4}}
\put(0,4){\line(1,0){4}}
\put(0,0){\line(0,1){4}}
\put(2,0){\line(0,1){4}}
\put(4,0){\line(0,1){4}}

\put(0,0){\line(1,1){2.1}}
\put(0,2){\line(1,1){2.1}}
\put(2,0){\line(1,1){2.1}}
\put(2,4){\line(1,-1){2.1}}

\put(1.5,-1.9){$\large A$}
}

\put(1.5,0){


\put(-1,4.3){1} \put(1.3,4.3){2} \put(4.3,4.3){3} 
\put(-1,2.3){4} \put(1.3,2.3){5} \put(4.3,2.3){6} 
\put(-1,-0.8){7} \put(1.3,-0.8){8} \put(4.3,-0.8){9} 

\put(0,4){\p} \put(2,4){\p} \put(4,4){\p} 
\put(0,2){\p} \put(2,2){\p} \put(4,2){\p} 
\put(0,0){\p} \put(2,0){\p} \put(4,0){\p} 

\put(0,0){\vector(1,0){2}}
\put(0,0){\line(1,0){4}}
\put(0,2){\line(1,0){4}}
\put(0,4){\line(1,0){4}}
\put(0,0){\line(0,1){4}}
\put(0,0){\vector(0,1){2}}
\put(2,0){\line(0,1){4}}
\put(4,0){\line(0,1){4}}

\put(0,0){\line(1,1){2.1}}
\put(0,2){\line(1,1){2.1}}
\put(2,0){\line(1,1){2.1}}
\put(2,4){\line(1,-1){2.1}}

\put(1.5,-1.9){$\large B$}
}

\put(10,0){


\put(-1,4.3){1} \put(1.3,4.3){2} \put(4.3,4.3){3} 
\put(-1,2.3){4} \put(1.3,2.3){5} \put(4.3,2.3){6} 
\put(-1,-0.8){7} \put(1.3,-0.8){8} \put(4.3,-0.8){9} 

\put(0,4){\p} \put(2,4){\p} \put(4,4){\p} 
\put(0,2){\p} \put(2,2){\p} \put(4,2){\p} 
\put(0,0){\p} \put(2,0){\p} \put(4,0){\p} 

\put(0,0){\vector(1,0){2}}
\put(0,0){\line(1,0){4}}
\put(0,2){\line(1,0){4}}
\put(0,4){\line(1,0){4}}
\put(0,0){\line(0,1){4}}
\put(2,2){\vector(0,-1){2}}
\put(2,0){\line(0,1){4}}
\put(0,2){\vector(0,-1){2}}
\put(0,2){\vector(1,0){2}}
\put(4,0){\line(0,1){4}}

\put(0,0){\vector(1,1){2}}

\put(0,2){\line(1,1){2.1}}
\put(2,0){\line(1,1){2.1}}
\put(2,4){\line(1,-1){2.1}}

\put(1.5,-1.9){$\large C$}
}
\end{picture}
\end{center}

\caption{Partial orientations $A,B,C$ of $T_{1}$.}
\label{ABC_T1}
\end{figure}

\begin{itemize}

\item B 47 (NPOC $B$, see
Figure \ref{ABC_T1}), C 4785, B 57 (NPOC $C$, see Figure \ref{ABC_T1}), C
5786, C 4562, C 2685, S 4257.

\item MC $C$, C 4752, C 2586, C 7568, S 4265.

\item MC $B$, B 45 (NPOC $D$, see Figure \ref{DEFG_T1}), C 7458, C 785, B
52 (NPOC $E$, see Figure \ref{DEFG_T1}), C 7524, C~8526, C 6245, S 7865.

\item MC $E$, C 7425, C 2456, C 7865, S 2685.

\item MC $D$, C 5478, B 57 (NPOC $F$, see Figure \ref{DEFG_T1}), C 5786, C
5742, C 2456, S 5268.

\item MC $F$, B 56 (NPOC $G$, see Figure \ref{DEFG_T1}), C 8756, C 5862, C
4562, S 7524.

\item MC $G$, C 6542, C 2586, S 7524.
\end{itemize}

Hence, the graph $T_{1}$ is not word-representable.

\begin{figure}[h]

\begin{center}
\begin{picture}(6,5)

\put(-12,0){


\put(-1,4.3){1} \put(1.3,4.3){2} \put(4.3,4.3){3} 
\put(-1,2.3){4} \put(1.3,2.3){5} \put(4.3,2.3){6} 
\put(-1,-0.8){7} \put(1.3,-0.8){8} \put(4.3,-0.8){9} 

\put(0,4){\p} \put(2,4){\p} \put(4,4){\p} 
\put(0,2){\p} \put(2,2){\p} \put(4,2){\p} 
\put(0,0){\p} \put(2,0){\p} \put(4,0){\p} 

\put(0,0){\vector(1,0){2}}
\put(0,0){\line(1,0){4}}
\put(0,2){\line(1,0){4}}
\put(0,4){\line(1,0){4}}
\put(0,0){\line(0,1){4}}
\put(0,0){\vector(0,1){2}}
\put(2,0){\line(0,1){4}}
\put(2,2){\vector(-1,0){2}}
\put(4,0){\line(0,1){4}}

\put(0,0){\line(1,1){2.1}}
\put(0,2){\line(1,1){2.1}}
\put(2,0){\line(1,1){2.1}}
\put(2,4){\line(1,-1){2.1}}

\put(1.5,-1.9){$\large D$}
}

\put(-3.5,0){


\put(-1,4.3){1} \put(1.3,4.3){2} \put(4.3,4.3){3} 
\put(-1,2.3){4} \put(1.3,2.3){5} \put(4.3,2.3){6} 
\put(-1,-0.8){7} \put(1.3,-0.8){8} \put(4.3,-0.8){9} 

\put(0,4){\p} \put(2,4){\p} \put(4,4){\p} 
\put(0,2){\p} \put(2,2){\p} \put(4,2){\p} 
\put(0,0){\p} \put(2,0){\p} \put(4,0){\p} 

\put(0,0){\vector(1,0){2}}
\put(0,0){\line(1,0){4}}
\put(0,2){\line(1,0){4}}
\put(0,2){\vector(0,1){2}}
\put(0,4){\line(1,0){4}}
\put(2,4){\vector(0,-1){2}}
\put(0,0){\line(0,1){4}}
\put(0,0){\vector(0,1){2}}
\put(2,0){\line(0,1){4}}
\put(2,0){\vector(1,0){2}}
\put(4,0){\line(0,1){4}}
\put(0,2){\vector(1,0){2}}
\put(2,0){\vector(0,1){2}}

\put(0,0){\vector(1,1){2}}
\put(0,2){\line(1,1){2.1}}
\put(2,0){\line(1,1){2.1}}
\put(2,4){\line(1,-1){2.1}}

\put(1.5,-1.9){$\large E$}
}

\put(5,0){


\put(-1,4.3){1} \put(1.3,4.3){2} \put(4.3,4.3){3} 
\put(-1,2.3){4} \put(1.3,2.3){5} \put(4.3,2.3){6} 
\put(-1,-0.8){7} \put(1.3,-0.8){8} \put(4.3,-0.8){9} 

\put(0,4){\p} \put(2,4){\p} \put(4,4){\p} 
\put(0,2){\p} \put(2,2){\p} \put(4,2){\p} 
\put(0,0){\p} \put(2,0){\p} \put(4,0){\p} 

\put(0,0){\vector(1,0){2}}
\put(0,0){\line(1,0){4}}
\put(0,2){\line(1,0){4}}
\put(0,4){\line(1,0){4}}
\put(0,0){\line(0,1){4}}
\put(2,2){\vector(0,-1){2}}
\put(2,2){\vector(-1,0){2}}
\put(2,0){\line(0,1){4}}
\put(0,0){\vector(0,1){2}}
\put(4,0){\line(0,1){4}}

\put(0,0){\vector(1,1){2}}

\put(0,2){\line(1,1){2.1}}
\put(2,0){\line(1,1){2.1}}
\put(2,4){\line(1,-1){2.1}}

\put(1.5,-1.9){$\large F$}
}

\put(13.5,0){


\put(-1,4.3){1} \put(1.3,4.3){2} \put(4.3,4.3){3} 
\put(-1,2.3){4} \put(1.3,2.3){5} \put(4.3,2.3){6} 
\put(-1,-0.8){7} \put(1.3,-0.8){8} \put(4.3,-0.8){9} 

\put(0,4){\p} \put(2,4){\p} \put(4,4){\p} 
\put(0,2){\p} \put(2,2){\p} \put(4,2){\p} 
\put(0,0){\p} \put(2,0){\p} \put(4,0){\p} 

\put(0,0){\vector(1,0){2}}
\put(0,0){\line(1,0){4}}
\put(0,2){\line(1,0){4}}
\put(0,4){\line(1,0){4}}
\put(0,0){\line(0,1){4}}
\put(2,2){\vector(0,-1){2}}
\put(2,2){\vector(-1,0){2}}
\put(2,0){\line(0,1){4}}
\put(4,2){\vector(-1,0){2}}
\put(0,0){\vector(0,1){2}}
\put(4,0){\line(0,1){4}}

\put(0,0){\vector(1,1){2}}

\put(0,2){\line(1,1){2.1}}
\put(2,0){\line(1,1){2.1}}
\put(2,4){\line(1,-1){2.1}}

\put(1.5,-1.9){$\large G$}
}

\end{picture}
\end{center}

\caption{Partial orientations $D,E,F,G$ of $T_{1}$.}
\label{DEFG_T1}
\end{figure}

\noindent 
{\bf Orienting $T_{2}$.} Name $A$ the first copy of $T_{2}$ with single edge orientation
$12$ and carry out the following operations. 

\begin{figure}[h]
\begin{center}
\begin{picture}(6,5)


\put(-12,0){

\put(-1,4.3){1} \put(1.3,4.3){2} \put(4.3,4.3){3} 
\put(-1,2.3){4} \put(1.3,1.2){5} \put(4.3,2.3){6} 
\put(-1,-0.8){7} \put(1.3,-0.8){8} \put(4.3,-0.8){9} 

\put(0,4){\p} \put(2,4){\p} \put(4,4){\p} 
\put(0,2){\p} \put(2,2){\p} \put(4,2){\p} 
\put(0,0){\p} \put(2,0){\p} \put(4,0){\p} 

\put(0,0){\line(1,0){4}}
\put(0,2){\line(1,0){4}}
\put(0,4){\line(1,0){4}}
\put(0,4){\vector(1,0){2}}
\put(0,0){\line(0,1){4}}
\put(2,0){\line(0,1){4}}
\put(4,0){\line(0,1){4}}

\put(0,2){\line(1,-1){2.1}}
\put(0,4){\line(1,-1){4}}
\put(2,2){\line(1,1){2.1}}

\put(1.5,-1.9){$\large A$}

}

\put(-3.5,0){

\put(-1,4.3){1} \put(1.3,4.3){2} \put(4.3,4.3){3} 
\put(-1,2.3){4} \put(1.3,1.2){5} \put(4.3,2.3){6} 
\put(-1,-0.8){7} \put(1.3,-0.8){8} \put(4.3,-0.8){9} 

\put(0,4){\p} \put(2,4){\p} \put(4,4){\p} 
\put(0,2){\p} \put(2,2){\p} \put(4,2){\p} 
\put(0,0){\p} \put(2,0){\p} \put(4,0){\p} 

\put(0,0){\line(1,0){4}}
\put(0,2){\line(1,0){4}}
\put(0,2){\vector(0,1){2}}
\put(0,4){\line(1,0){4}}
\put(0,4){\vector(1,0){2}}
\put(0,0){\line(0,1){4}}
\put(2,0){\line(0,1){4}}
\put(4,0){\line(0,1){4}}

\put(0,2){\line(1,-1){2.1}}
\put(0,4){\line(1,-1){4}}
\put(2,2){\line(1,1){2.1}}

\put(1.5,-1.9){$\large B$}

}


\put(5,0){

\put(-1,4.3){1} \put(1.3,4.3){2} \put(4.3,4.3){3} 
\put(-1,2.3){4} \put(1.3,1.2){5} \put(4.3,2.3){6} 
\put(-1,-0.8){7} \put(1.3,-0.8){8} \put(4.3,-0.8){9} 

\put(0,4){\p} \put(2,4){\p} \put(4,4){\p} 
\put(0,2){\p} \put(2,2){\p} \put(4,2){\p} 
\put(0,0){\p} \put(2,0){\p} \put(4,0){\p} 

\put(0,0){\line(1,0){4}}
\put(0,2){\line(1,0){4}}
\put(0,4){\line(1,0){4}}
\put(0,4){\vector(1,0){2}}
\put(0,4){\vector(0,-1){2}}
\put(0,0){\line(0,1){4}}
\put(2,0){\line(0,1){4}}
\put(2,0){\vector(-1,1){2}}
\put(4,0){\line(0,1){4}}

\put(0,2){\line(1,-1){2.1}}
\put(0,4){\line(1,-1){4}}
\put(2,2){\line(1,1){2.1}}

\put(1.5,-1.9){$\large C$}

}

\put(13.5,0){

\put(-1,4.3){1} \put(1.3,4.3){2} \put(4.3,4.3){3} 
\put(-1,2.3){4} \put(1.3,1.2){5} \put(4.3,2.3){6} 
\put(-1,-0.8){7} \put(1.3,-0.8){8} \put(4.3,-0.8){9} 

\put(0,4){\p} \put(2,4){\p} \put(4,4){\p} 
\put(0,2){\p} \put(2,2){\p} \put(4,2){\p} 
\put(0,0){\p} \put(2,0){\p} \put(4,0){\p} 

\put(0,0){\line(1,0){4}}
\put(0,2){\line(1,0){4}}
\put(0,4){\line(1,0){4}}
\put(0,4){\vector(1,0){2}}
\put(0,4){\vector(0,-1){2}}
\put(0,0){\line(0,1){4}}
\put(2,0){\line(0,1){4}}
\put(4,0){\line(0,1){4}}

\put(0,2){\vector(1,-1){2}}
\put(0,4){\line(1,-1){4}}
\put(0,4){\vector(1,-1){2}}
\put(2,2){\line(1,1){2.1}}
\put(2,2){\vector(0,-1){2}}
\put(4,4){\vector(-1,-1){2}}

\put(1.5,-1.9){$\large D$}

}

\end{picture}
\end{center}

\caption{Partial orientations $A,B,C,D$ of $T_{2}$.}
\label{ABCD_T2}
\end{figure}

\begin{figure}[h]

\begin{center}
\begin{picture}(6,5)


\put(-12,0){

\put(-1,4.3){1} \put(1.3,4.3){2} \put(4.3,4.3){3} 
\put(-1,2.3){4} \put(1.3,1.2){5} \put(4.3,2.3){6} 
\put(-1,-0.8){7} \put(1.3,-0.8){8} \put(4.3,-0.8){9} 

\put(0,4){\p} \put(2,4){\p} \put(4,4){\p} 
\put(0,2){\p} \put(2,2){\p} \put(4,2){\p} 
\put(0,0){\p} \put(2,0){\p} \put(4,0){\p} 

\put(0,0){\line(1,0){4}}
\put(0,2){\line(1,0){4}}
\put(0,4){\line(1,0){4}}
\put(0,4){\vector(1,0){2}}
\put(0,4){\vector(0,-1){2}}
\put(0,0){\line(0,1){4}}
\put(2,0){\line(0,1){4}}
\put(2,4){\vector(1,0){2}}
\put(4,0){\line(0,1){4}}
\put(4,2){\vector(-1,0){2}}

\put(0,2){\vector(1,-1){2}}
\put(0,4){\line(1,-1){4}}
\put(0,4){\vector(1,-1){2}}
\put(2,2){\line(1,1){2.1}}
\put(2,2){\vector(1,1){2}}
\put(2,2){\vector(0,-1){2}}

\put(1.5,-1.9){$\large E$}

}

\put(-3.5,0){

\put(-1,4.3){1} \put(1.3,4.3){2} \put(4.3,4.3){3} 
\put(-1,2.3){4} \put(1.3,1.2){5} \put(4.3,2.3){6} 
\put(-1,-0.8){7} \put(1.3,-0.8){8} \put(4.3,-0.8){9} 

\put(0,4){\p} \put(2,4){\p} \put(4,4){\p} 
\put(0,2){\p} \put(2,2){\p} \put(4,2){\p} 
\put(0,0){\p} \put(2,0){\p} \put(4,0){\p} 

\put(0,0){\line(1,0){4}}
\put(0,2){\line(1,0){4}}
\put(0,4){\line(1,0){4}}
\put(0,4){\vector(1,0){2}}
\put(0,4){\vector(0,-1){2}}
\put(0,0){\line(0,1){4}}
\put(2,0){\line(0,1){4}}
\put(4,0){\line(0,1){4}}
\put(2,4){\vector(1,0){2}}

\put(0,2){\vector(1,-1){2}}
\put(0,4){\line(1,-1){4}}
\put(0,4){\vector(1,-1){2}}
\put(2,2){\vector(1,1){2}}
\put(2,2){\vector(-1,0){2}}
\put(2,2){\vector(1,0){2}}
\put(2,2){\vector(0,-1){2}}

\put(1.5,-1.9){$\large F$}

}

\put(5,0){

\put(-1,4.3){1} \put(1.3,4.3){2} \put(4.3,4.3){3} 
\put(-1,2.3){4} \put(1.3,1.2){5} \put(4.3,2.3){6} 
\put(-1,-0.8){7} \put(1.3,-0.8){8} \put(4.3,-0.8){9} 

\put(0,4){\p} \put(2,4){\p} \put(4,4){\p} 
\put(0,2){\p} \put(2,2){\p} \put(4,2){\p} 
\put(0,0){\p} \put(2,0){\p} \put(4,0){\p} 

\put(0,0){\line(1,0){4}}
\put(0,2){\line(1,0){4}}
\put(0,4){\line(1,0){4}}
\put(0,4){\vector(1,0){2}}
\put(0,4){\vector(0,-1){2}}
\put(0,0){\line(0,1){4}}
\put(2,0){\line(0,1){4}}
\put(4,0){\line(0,1){4}}

\put(0,2){\vector(1,-1){2}}
\put(0,4){\line(1,-1){4}}
\put(0,4){\vector(1,-1){2}}
\put(2,2){\line(1,1){2.1}}
\put(2,2){\vector(1,-1){2}}
\put(2,2){\vector(0,-1){2}}
\put(4,4){\vector(-1,-1){2}}
\put(4,4){\vector(-1,0){2}}

\put(1.5,-1.9){$\large G$}

}

\put(13.5,0){

\put(-1,4.3){1} \put(1.3,4.3){2} \put(4.3,4.3){3} 
\put(-1,2.3){4} \put(1.3,1.2){5} \put(4.3,2.3){6} 
\put(-1,-0.8){7} \put(1.3,-0.8){8} \put(4.3,-0.8){9} 

\put(0,4){\p} \put(2,4){\p} \put(4,4){\p} 
\put(0,2){\p} \put(2,2){\p} \put(4,2){\p} 
\put(0,0){\p} \put(2,0){\p} \put(4,0){\p} 

\put(0,0){\line(1,0){4}}
\put(0,2){\line(1,0){4}}
\put(0,4){\line(1,0){4}}
\put(0,4){\vector(1,0){2}}
\put(0,4){\vector(0,-1){2}}
\put(0,0){\line(0,1){4}}
\put(2,0){\line(0,1){4}}
\put(4,0){\line(0,1){4}}
\put(4,2){\vector(0,-1){2}}

\put(0,2){\vector(1,-1){2}}
\put(0,4){\line(1,-1){4}}
\put(0,4){\vector(1,-1){2}}
\put(2,2){\line(1,1){2.1}}
\put(2,2){\vector(1,-1){2}}
\put(2,2){\vector(0,1){2}}
\put(2,2){\vector(0,-1){2}}
\put(4,4){\vector(-1,-1){2}}
\put(4,4){\vector(-1,0){2}}
\put(4,4){\vector(0,-1){2}}

\put(1.5,-1.9){$\large H$}

}

\end{picture}
\end{center}

\caption{Partial orientations $E,F,G,H$ of $T_{2}$.}
\label{EFGH_T2}
\end{figure}

\begin{itemize}

\item B 14 (NPOC $B$, see
Figure \ref{ABCD_T2}), B 48 (NPOC $C$, see Figure \ref{ABCD_T2}), C 1485,
B 53 (NPOC $D$, see Figure \ref{ABCD_T2}), C 1532, B 56 (NPOC $E$,
see Figure \ref{EFGH_T2}), B 45 (NPOC $F$, see Figure \ref{EFGH_T2}), C
1452, C 2563, C~5369, C 4598, S 5896.

\item MC $F$, C 4598, C 1254, C 2365, C 5369, S 5698.

\item MC $E$, C 2365, C 1254, C 5698, C 5963, S 4598.

\item MC $D$, C 1235, B 95 (NPOC $G$, see Figure \ref{EFGH_T2}), C 4598,
C 4587, C 1254, C 2365, C 3695, S 9658.

\item MC $G$, C 3596, B 25 (NPOC $H$, see Figure \ref{EFGH_T2}), C 1254,
C 4598, C 5698, S 3256.

\item MC $H$, C 1254, C 4598, C 5698, S 3652.

\item MC $C$, B 25 (NPOC $I$, see Figure \ref{IJKL_T2}), C 1254, C 485,
C 1584, C 4598, B 23 (NPOC $J$, see Figure \ref{IJKL_T2}), C 1532, C
9536, C 8965, S 2563.

\item MC $J$, C 2365, C 5698, C 1532, S 3695.

\item MC $I$, C 1254, B 58 (NPOC $K$, see Figure \ref{IJKL_T2}), C 1485,
C 5123, B 56 (NPOC $L$, see Figure \ref{IJKL_T2}), C~5236, C 5369, C
5698, S 5984.

\item MC $L$, C 2365, C 5896, C 5963, S 5984.

\item MC $K$, C 1485, B 23 (NPOC $M$, see Figure \ref{MNOP_T2}), C 1235,
C 5236, C 8569, C 5963, S 8954.

\item MC $M$, C 1235, B 59 (NPOC $N$, see Figure \ref{MNOP_T2}), C 3596,
C 3652, C 5984, S 8569.

\item MC $N$, C 9548, C 9856, C 9536, S 3652.

\item MC $B$, C 1254, B 15 (NPOC $O$, see Figure \ref{MNOP_T2}), C 4158,
B 98 (NPOC $P$, see Figure \ref{MNOP_T2}), C 4598, C 9856, C 2365, C
3695, S 1532.

\item MC $P$, C 4598, B 53 (NPOC $Q$, see Figure \ref{QR_T2}), C 1235,
C 2365, C 5698, S 5963. 

\item MC $Q$, C 3596, C 3256, S 8569. 

\item MC $O$, C 1235, B 95 (NPOC $R$, see Figure \ref{QR_T2}), C 3695,
C 2365, C 9658, C 4598, S 4851.

\item MC $R$, C 4598, C 1485, C 5896, C 5369, S 5632.
\end{itemize}

Hence, the graph $T_{2}$ is not word-representable. \end{proof}

\begin{figure}[h]

\begin{center}
\begin{picture}(6,5)


\put(-12,0){

\put(-1,4.3){1} \put(1.3,4.3){2} \put(4.3,4.3){3} 
\put(-1,2.3){4} \put(1.3,1.2){5} \put(4.3,2.3){6} 
\put(-1,-0.8){7} \put(1.3,-0.8){8} \put(4.3,-0.8){9} 

\put(0,4){\p} \put(2,4){\p} \put(4,4){\p} 
\put(0,2){\p} \put(2,2){\p} \put(4,2){\p} 
\put(0,0){\p} \put(2,0){\p} \put(4,0){\p} 

\put(0,0){\line(1,0){4}}
\put(0,2){\line(1,0){4}}
\put(0,4){\line(1,0){4}}
\put(0,4){\vector(1,0){2}}
\put(0,4){\vector(0,-1){2}}
\put(0,0){\line(0,1){4}}
\put(2,0){\line(0,1){4}}
\put(2,0){\vector(-1,1){2}}
\put(4,0){\line(0,1){4}}

\put(0,2){\line(1,-1){2.1}}
\put(0,4){\line(1,-1){4}}
\put(2,2){\line(1,1){2.1}}
\put(2,2){\vector(0,1){2}}

\put(1.5,-1.9){$\large I$}

}

\put(-3.5,0){

\put(-1,4.3){1} \put(1.3,4.3){2} \put(4.3,4.3){3} 
\put(-1,2.3){4} \put(1.3,1.2){5} \put(4.3,2.3){6} 
\put(-1,-0.8){7} \put(1.3,-0.8){8} \put(4.3,-0.8){9} 

\put(0,4){\p} \put(2,4){\p} \put(4,4){\p} 
\put(0,2){\p} \put(2,2){\p} \put(4,2){\p} 
\put(0,0){\p} \put(2,0){\p} \put(4,0){\p} 

\put(0,0){\line(1,0){4}}
\put(0,2){\line(1,0){4}}
\put(0,2){\vector(1,0){2}}
\put(0,4){\line(1,0){4}}
\put(0,4){\vector(1,0){2}}
\put(0,4){\vector(0,-1){2}}
\put(0,0){\line(0,1){4}}
\put(2,0){\line(0,1){4}}
\put(2,0){\vector(-1,1){2}}
\put(2,0){\vector(0,1){2}}
\put(2,0){\vector(1,0){2}}
\put(4,0){\line(0,1){4}}
\put(4,0){\vector(-1,1){2}}
\put(2,4){\vector(0,-1){2}}

\put(0,2){\line(1,-1){2.1}}
\put(0,4){\line(1,-1){4}}
\put(0,4){\vector(1,-1){2}}
\put(2,2){\line(1,1){2.1}}
\put(4,4){\vector(-1,0){2}}

\put(1.5,-1.9){$\large J$}

}

\put(5,0){

\put(-1,4.3){1} \put(1.3,4.3){2} \put(4.3,4.3){3} 
\put(-1,2.3){4} \put(1.3,1.2){5} \put(4.3,2.3){6} 
\put(-1,-0.8){7} \put(1.3,-0.8){8} \put(4.3,-0.8){9} 

\put(0,4){\p} \put(2,4){\p} \put(4,4){\p} 
\put(0,2){\p} \put(2,2){\p} \put(4,2){\p} 
\put(0,0){\p} \put(2,0){\p} \put(4,0){\p} 

\put(0,0){\line(1,0){4}}
\put(0,2){\line(1,0){4}}
\put(0,4){\line(1,0){4}}
\put(0,4){\vector(1,0){2}}
\put(0,4){\vector(0,-1){2}}
\put(0,0){\line(0,1){4}}
\put(2,0){\line(0,1){4}}
\put(2,0){\vector(0,1){2}}
\put(2,0){\vector(-1,1){2}}
\put(4,0){\line(0,1){4}}

\put(0,2){\line(1,-1){2.1}}
\put(0,4){\line(1,-1){4}}
\put(2,2){\line(1,1){2.1}}
\put(2,2){\vector(-1,0){2}}
\put(2,2){\vector(0,1){2}}

\put(1.5,-1.9){$\large K$}

}

\put(13.5,0){

\put(-1,4.3){1} \put(1.3,4.3){2} \put(4.3,4.3){3} 
\put(-1,2.3){4} \put(1.3,1.2){5} \put(4.3,2.3){6} 
\put(-1,-0.8){7} \put(1.3,-0.8){8} \put(4.3,-0.8){9} 

\put(0,4){\p} \put(2,4){\p} \put(4,4){\p} 
\put(0,2){\p} \put(2,2){\p} \put(4,2){\p} 
\put(0,0){\p} \put(2,0){\p} \put(4,0){\p} 

\put(0,0){\line(1,0){4}}
\put(0,2){\line(1,0){4}}
\put(0,4){\line(1,0){4}}
\put(0,4){\vector(1,0){2}}
\put(0,4){\vector(0,-1){2}}
\put(0,0){\line(0,1){4}}
\put(2,0){\line(0,1){4}}
\put(2,0){\vector(-1,1){2}}
\put(4,0){\line(0,1){4}}
\put(4,2){\vector(-1,0){2}}

\put(0,2){\line(1,-1){2.1}}
\put(2,2){\vector(0,1){2}}
\put(2,2){\vector(-1,1){2}}
\put(2,2){\vector(-1,0){2}}
\put(2,2){\vector(0,-1){2}}
\put(0,4){\line(1,-1){4}}
\put(2,2){\vector(1,1){2}}
\put(4,4){\vector(-1,0){2}}

\put(1.5,-1.9){$\large L$}

}

\end{picture}
\end{center}

\caption{Partial orientations $I,J,K,L$ of $T_{2}$.}
\label{IJKL_T2}
\end{figure}

\begin{figure}[h]

\begin{center}
\begin{picture}(6,5)


\put(-12,0){

\put(-1,4.3){1} \put(1.3,4.3){2} \put(4.3,4.3){3} 
\put(-1,2.3){4} \put(1.3,1.2){5} \put(4.3,2.3){6} 
\put(-1,-0.8){7} \put(1.3,-0.8){8} \put(4.3,-0.8){9} 

\put(0,4){\p} \put(2,4){\p} \put(4,4){\p} 
\put(0,2){\p} \put(2,2){\p} \put(4,2){\p} 
\put(0,0){\p} \put(2,0){\p} \put(4,0){\p} 

\put(0,0){\line(1,0){4}}
\put(0,2){\line(1,0){4}}
\put(0,4){\line(1,0){4}}
\put(0,4){\vector(1,0){2}}
\put(0,4){\vector(0,-1){2}}
\put(0,0){\line(0,1){4}}
\put(2,0){\line(0,1){4}}
\put(2,0){\vector(-1,1){2}}
\put(2,0){\vector(0,1){2}}
\put(4,0){\line(0,1){4}}

\put(0,2){\line(1,-1){2.1}}
\put(0,4){\line(1,-1){4}}
\put(0,4){\vector(1,-1){2}}
\put(2,2){\line(1,1){2.1}}
\put(2,2){\vector(-1,0){2}}
\put(2,2){\vector(0,1){2}}
\put(4,4){\vector(-1,0){2}}

\put(1.5,-1.9){$\large M$}

}

\put(-3.5,0){

\put(-1,4.3){1} \put(1.3,4.3){2} \put(4.3,4.3){3} 
\put(-1,2.3){4} \put(1.3,1.2){5} \put(4.3,2.3){6} 
\put(-1,-0.8){7} \put(1.3,-0.8){8} \put(4.3,-0.8){9} 

\put(0,4){\p} \put(2,4){\p} \put(4,4){\p} 
\put(0,2){\p} \put(2,2){\p} \put(4,2){\p} 
\put(0,0){\p} \put(2,0){\p} \put(4,0){\p} 

\put(0,0){\line(1,0){4}}
\put(0,2){\line(1,0){4}}
\put(0,4){\line(1,0){4}}
\put(0,4){\vector(1,0){2}}
\put(0,4){\vector(0,-1){2}}
\put(0,0){\line(0,1){4}}
\put(2,0){\line(0,1){4}}
\put(2,0){\vector(-1,1){2}}
\put(2,0){\vector(0,1){2}}
\put(4,0){\line(0,1){4}}
\put(4,0){\vector(-1,1){2}}

\put(0,2){\line(1,-1){2.1}}
\put(0,4){\line(1,-1){4}}
\put(0,4){\vector(1,-1){2}}
\put(2,2){\line(1,1){2.1}}
\put(2,2){\vector(-1,0){2}}
\put(2,2){\vector(0,1){2}}
\put(4,4){\vector(-1,0){2}}
\put(4,4){\vector(-1,-1){2}}

\put(1.5,-1.9){$\large N$}

}

\put(5,0){

\put(-1,4.3){1} \put(1.3,4.3){2} \put(4.3,4.3){3} 
\put(-1,2.3){4} \put(1.3,1.2){5} \put(4.3,2.3){6} 
\put(-1,-0.8){7} \put(1.3,-0.8){8} \put(4.3,-0.8){9} 

\put(0,4){\p} \put(2,4){\p} \put(4,4){\p} 
\put(0,2){\p} \put(2,2){\p} \put(4,2){\p} 
\put(0,0){\p} \put(2,0){\p} \put(4,0){\p} 

\put(0,0){\line(1,0){4}}
\put(0,2){\line(1,0){4}}
\put(0,2){\vector(0,1){2}}
\put(0,2){\vector(1,0){2}}
\put(0,4){\line(1,0){4}}
\put(0,4){\vector(1,0){2}}
\put(0,0){\line(0,1){4}}
\put(2,0){\line(0,1){4}}
\put(4,0){\line(0,1){4}}

\put(0,2){\line(1,-1){2.1}}
\put(0,4){\line(1,-1){4}}
\put(2,2){\line(1,1){2.1}}
\put(2,2){\vector(0,1){2}}
\put(2,2){\vector(-1,1){2}}

\put(1.5,-1.9){$\large O$}

}

\put(13.5,0){

\put(-1,4.3){1} \put(1.3,4.3){2} \put(4.3,4.3){3} 
\put(-1,2.3){4} \put(1.3,1.2){5} \put(4.3,2.3){6} 
\put(-1,-0.8){7} \put(1.3,-0.8){8} \put(4.3,-0.8){9} 

\put(0,4){\p} \put(2,4){\p} \put(4,4){\p} 
\put(0,2){\p} \put(2,2){\p} \put(4,2){\p} 
\put(0,0){\p} \put(2,0){\p} \put(4,0){\p} 

\put(0,0){\line(1,0){4}}
\put(0,2){\line(1,0){4}}
\put(0,2){\vector(0,1){2}}
\put(0,2){\vector(1,0){2}}
\put(0,4){\line(1,0){4}}
\put(0,4){\vector(1,0){2}}
\put(0,0){\line(0,1){4}}
\put(2,0){\line(0,1){4}}
\put(2,0){\vector(1,0){2}}
\put(2,0){\vector(0,1){2}}
\put(4,0){\line(0,1){4}}

\put(0,2){\vector(1,-1){2}}
\put(0,4){\line(1,-1){4}}
\put(0,4){\vector(1,-1){2}}
\put(2,2){\line(1,1){2.1}}
\put(2,2){\vector(0,1){2}}

\put(1.5,-1.9){$\large P$}

}

\end{picture}
\end{center}

\caption{Partial orientations $M,N,O,P$ of $T_{2}$.}
\label{MNOP_T2}
\end{figure}

\begin{figure}[h]

\begin{center}
\begin{picture}(6,5)


\put(-5,0){

\put(-1,4.3){1} \put(1.3,4.3){2} \put(4.3,4.3){3} 
\put(-1,2.3){4} \put(1.3,1.2){5} \put(4.3,2.3){6} 
\put(-1,-0.8){7} \put(1.3,-0.8){8} \put(4.3,-0.8){9} 

\put(0,4){\p} \put(2,4){\p} \put(4,4){\p} 
\put(0,2){\p} \put(2,2){\p} \put(4,2){\p} 
\put(0,0){\p} \put(2,0){\p} \put(4,0){\p} 

\put(0,0){\line(1,0){4}}
\put(0,2){\line(1,0){4}}
\put(0,4){\line(1,0){4}}
\put(0,4){\vector(1,0){2}}
\put(0,0){\line(0,1){4}}
\put(2,0){\line(0,1){4}}
\put(2,0){\vector(1,0){2}}
\put(2,0){\vector(0,1){2}}
\put(4,0){\line(0,1){4}}

\put(0,2){\vector(1,-1){2}}
\put(0,2){\vector(1,0){2}}
\put(0,2){\vector(0,1){2}}
\put(0,4){\line(1,-1){4}}
\put(0,4){\vector(1,-1){2}}
\put(2,2){\line(1,1){2.1}}
\put(2,2){\vector(0,1){2}}
\put(2,2){\vector(1,-1){2}}
\put(4,4){\vector(-1,-1){2}}

\put(1.5,-1.9){$\large Q$}

}


\put(7,0){

\put(-1,4.3){1} \put(1.3,4.3){2} \put(4.3,4.3){3} 
\put(-1,2.3){4} \put(1.3,1.2){5} \put(4.3,2.3){6} 
\put(-1,-0.8){7} \put(1.3,-0.8){8} \put(4.3,-0.8){9} 

\put(0,4){\p} \put(2,4){\p} \put(4,4){\p} 
\put(0,2){\p} \put(2,2){\p} \put(4,2){\p} 
\put(0,0){\p} \put(2,0){\p} \put(4,0){\p} 

\put(0,0){\line(1,0){4}}
\put(0,2){\line(1,0){4}}
\put(0,2){\vector(1,0){2}}
\put(0,2){\vector(0,1){2}}
\put(0,4){\line(1,0){4}}
\put(0,4){\vector(1,0){2}}
\put(0,0){\line(0,1){4}}
\put(2,0){\line(0,1){4}}
\put(4,0){\line(0,1){4}}

\put(0,2){\line(1,-1){2.1}}
\put(0,4){\line(1,-1){4}}
\put(2,2){\line(1,1){2.1}}
\put(2,2){\vector(1,1){2}}
\put(2,2){\vector(-1,1){2}}
\put(2,2){\vector(1,-1){2}}
\put(2,2){\vector(0,1){2}}
\put(4,4){\vector(-1,0){2}}

\put(1.5,-1.9){$\large R$}

}

\end{picture}
\end{center}

\caption{Partial orientations $Q,R$ of $T_{2}$.}
\label{QR_T2}
\end{figure}

\section{Triangulations of a polyomino}\label{sec3}

In this section we first consider triangulations of a grid graph, and then generalize our arguments to triangulations of a convex polyomino. The main theorem in this paper, Theorem~\ref{main-thm}, is to be proved in Subsection~\ref{tr-conv-polyomino}. We will also show in Subsection~\ref{triangArbitrary} that Theorem~\ref{main-thm} does not hold in the case of an arbitrary polyomino.  

\subsection{Triangulations of a grid graph}

Let $S$ be the set of eight graphs formed by rotations of $T_1$ and $T_2$ in Figure~\ref{non-repr-triang} by degrees multiple to $90\,^{\circ}$.

\begin{lemma}\label{lemma-grid} A triangulation $T$ of a grid graph is $3$-colorable if and only if it does not contain a graph from $S$ as an induced subgraph.  \end{lemma}

\begin{proof} If $T$ contains a graph from $S$ as an induced subgraph, then it is obviously not $3$-colorable. 

For the opposite directions, suppose that $T$ is not $3$-colorable. We note that fixing colors of the left-most top vertex in $T$ and the vertex right below it determines uniquely colours in the top two rows of $T$ (a row is a horizontal path) if we are to use colors in $\{1,2,3\}$ and keep all other vertices of $T$ uncolored.  We continue to color all other vertices of $T$, row by row, from left to right using any of the available colors in $\{1,2,3\}$. At some point, color 4 must be used ($T$ is not $3$-colorable). Let $v$ be the first vertex colored by 4. There are only three possible different situations when this can happen, which are presented in Figure~\ref{conflict-situations} (numbers in this figure are colors). In that figure, the shaded area indicates schematically already colored vertices of $T$, the question mark shows a still non-colored vertex, and the colors adjacent to $v$ are fixed in a particular way without loss of generality (we can re-name already used colors if needed). A particular property in all cases is that among the colors of neighbours of $v$, we meet all the colors in $\{1,2,3\}$. Also,  by our procedure, $v$ must be in row $i$ from above, where $i\geq 3$.\\

\begin{figure}[h]
\begin{center}
\begin{picture}(8,7)


\put(-12,0){

\put(-2.5,2.5){$C_1$=}

\put(2.5,1.5){{\tiny 1}} 
\put(2.5,4.1){{\tiny 2}} 
\put(5.2,1.5){{\tiny 4}} 
\put(5.2,3.5){{\tiny 3}} 

\put(3,2){\p} \put(5,2){\p} \put(3,4){\p} \put(5,4){\p} 

\put(0,0){\line(1,0){9}}
\put(0,6){\line(1,0){9}}
\put(0,2){\line(1,0){5}}
\put(3,4){\line(1,0){6}}
\put(3,2){\line(0,1){2}}
\put(5,2){\line(0,1){2}}
\put(3,4){\line(1,-1){2.1}}
\put(0,0){\line(0,1){6}}
\put(9,0){\line(0,1){6}}

\multiput(0,2)(0.5,0.5){8}{\line(1,1){0.2}}
\multiput(1,2)(0.5,0.5){8}{\line(1,1){0.2}}
\multiput(2,2)(0.5,0.5){2}{\line(1,1){0.2}}
\multiput(0,3)(0.5,0.5){6}{\line(1,1){0.2}}
\multiput(0,4)(0.5,0.5){4}{\line(1,1){0.2}}
\multiput(0,5)(0.5,0.5){2}{\line(1,1){0.2}}
\multiput(4,4)(0.5,0.5){4}{\line(1,1){0.2}}
\multiput(5,4)(0.5,0.5){4}{\line(1,1){0.2}}
\multiput(6,4)(0.5,0.5){4}{\line(1,1){0.2}}
\multiput(7,4)(0.5,0.5){4}{\line(1,1){0.2}}
\multiput(8,4)(0.5,0.5){2}{\line(1,1){0.2}}

}


\put(1,0){

\put(-2.5,2.5){$C_2$=}

\put(2.5,1.5){{\tiny 1}} 
\put(2.5,4.1){{\tiny 2}} 
\put(5.2,1.5){{\tiny 4}} 
\put(5.2,3.5){{\tiny 1}} 
\put(7.2,3.5){{\tiny 3}} 
\put(7.2,1.5){{\tiny ?}} 

\put(3,2){\p} \put(5,2){\p} \put(3,4){\p} \put(5,4){\p} \put(7,4){\p} \put(7,2){\p} 

\put(5,2){\line(1,1){2.1}}
\put(0,0){\line(1,0){9}}
\put(0,6){\line(1,0){9}}
\put(0,2){\line(1,0){7}}
\put(3,4){\line(1,0){6}}
\put(3,2){\line(0,1){2}}
\put(5,2){\line(0,1){2}}
\put(3,4){\line(1,-1){2.1}}
\put(0,0){\line(0,1){6}}
\put(9,0){\line(0,1){6}}
\put(7,2){\line(0,1){2}}

\multiput(0,2)(0.5,0.5){8}{\line(1,1){0.2}}
\multiput(1,2)(0.5,0.5){8}{\line(1,1){0.2}}
\multiput(2,2)(0.5,0.5){2}{\line(1,1){0.2}}
\multiput(0,3)(0.5,0.5){6}{\line(1,1){0.2}}
\multiput(0,4)(0.5,0.5){4}{\line(1,1){0.2}}
\multiput(0,5)(0.5,0.5){2}{\line(1,1){0.2}}
\multiput(4,4)(0.5,0.5){4}{\line(1,1){0.2}}
\multiput(5,4)(0.5,0.5){4}{\line(1,1){0.2}}
\multiput(6,4)(0.5,0.5){4}{\line(1,1){0.2}}
\multiput(7,4)(0.5,0.5){4}{\line(1,1){0.2}}
\multiput(8,4)(0.5,0.5){2}{\line(1,1){0.2}}

}


\put(14,0){

\put(-2.5,2.5){$C_3$=}

\put(2.5,1.5){{\tiny 1}} 
\put(2.5,4.1){{\tiny 3}} 
\put(5.2,1.5){{\tiny 4}} 
\put(5.2,3.5){{\tiny 2}} 
\put(7.2,3.5){{\tiny 3}} 
\put(7.2,1.5){{\tiny ?}} 

\put(3,2){\p} \put(5,2){\p} \put(3,4){\p} \put(5,4){\p} \put(7,4){\p} \put(7,2){\p} 

\put(5,2){\line(1,1){2.1}}
\put(0,0){\line(1,0){9}}
\put(0,6){\line(1,0){9}}
\put(0,2){\line(1,0){7}}
\put(3,4){\line(1,0){6}}
\put(3,2){\line(0,1){2}}
\put(5,2){\line(0,1){2}}
\put(3,2){\line(1,1){2.1}}
\put(0,0){\line(0,1){6}}
\put(9,0){\line(0,1){6}}
\put(7,2){\line(0,1){2}}

\multiput(0,2)(0.5,0.5){8}{\line(1,1){0.2}}
\multiput(1,2)(0.5,0.5){8}{\line(1,1){0.2}}
\multiput(2,2)(0.5,0.5){2}{\line(1,1){0.2}}
\multiput(0,3)(0.5,0.5){6}{\line(1,1){0.2}}
\multiput(0,4)(0.5,0.5){4}{\line(1,1){0.2}}
\multiput(0,5)(0.5,0.5){2}{\line(1,1){0.2}}
\multiput(4,4)(0.5,0.5){4}{\line(1,1){0.2}}
\multiput(5,4)(0.5,0.5){4}{\line(1,1){0.2}}
\multiput(6,4)(0.5,0.5){4}{\line(1,1){0.2}}
\multiput(7,4)(0.5,0.5){4}{\line(1,1){0.2}}
\multiput(8,4)(0.5,0.5){2}{\line(1,1){0.2}}

}

\end{picture}
\caption{Three possible cases of appearance of color 4 in coloring of $T$.} \label{conflict-situations}
\end{center}
\end{figure}
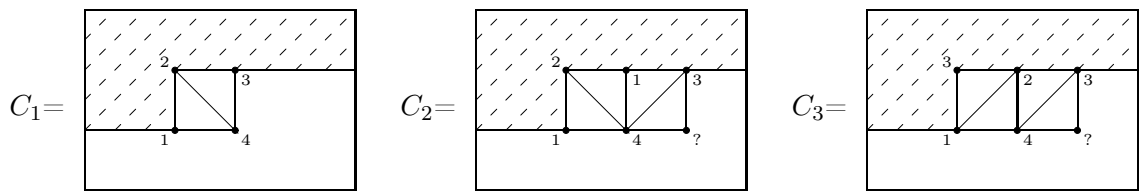

\noindent
{\bf Case 1: Triangulation $C_1$ in Figure~\ref{conflict-situations}}. Note that the vertex colored by 1 must be in column $i$ (from left to right), where $i\geq 2$; if the vertex would be in column 1, there would be no need to color it by $1$ in our procedure --- color $3$ could be used contradicting the assumption (we would be in conditions of Case 2 to be considered below). Thus, $T$ must contain, as an induced subgraph, a $3\times 3$ grid graph triangulation with the right-most bottom vertex being $v$. We have four possible subcases depending on the colors of the vertices indicated by white circles in Figure~\ref{case1-subcases}, which allows us, in each case, to partially recover the triangulation of the $3\times 3$ grid graph involved, as well as some of vertices' colors. However, in each of the cases, there is a unique way to complete the triangulation, namely, by joining the vertices colored by 1 and 3. Indeed, if the vertex colored by 2 is connected to a ? vertex, then the ? vertex must be colored by 4 contradicting the fact that $v$ was the first vertex colored by 4 in our colouring procedure. We see that in each case, the triangulation belongs to $S$.\\

\begin{figure}[h]
\begin{center}
\begin{picture}(10,4.5)

\put(-12,0){


\put(-1,4.3){3} \put(1.3,4.3){1} \put(4.3,4.3){?} 
\put(-1,2.3){1} \put(2.2,2.3){2} \put(4.3,2.3){3} 
\put(-1,-0.8){3} \put(1.3,-0.8){1} \put(4.3,-0.8){4} 

\put(0,4){\p} \put(1.8,3.8){$\circ$} \put(4,4){\p} 
\put(-0.2,1.8){$\circ$} \put(2,2){\p} \put(4,2){\p} 
\put(0,0){\p} \put(2,0){\p} \put(4,0){\p} 

\put(0,0){\line(1,0){4}}
\put(0,2){\line(1,0){4}}
\put(0,4){\line(1,0){4}}
\put(0,0){\line(0,1){4}}
\put(2,0){\line(0,1){4}}
\put(4,0){\line(0,1){4}}

\put(0,0){\line(1,1){2.1}}
\put(0,4){\line(1,-1){2.1}}
\put(2,2){\line(1,-1){2.1}}

}

\put(-3,0){


\put(-1,4.3){?} \put(1.3,4.3){1} \put(4.3,4.3){?} 
\put(-1,2.3){3} \put(1.3,2.3){2} \put(4.3,2.3){3} 
\put(-1,-0.8){?} \put(1.3,-0.8){1} \put(4.3,-0.8){4} 

\put(0,4){\p} \put(1.8,3.8){$\circ$} \put(4,4){\p} 
\put(-0.2,1.8){$\circ$} \put(2,2){\p} \put(4,2){\p} 
\put(0,0){\p} \put(2,0){\p} \put(4,0){\p} 

\put(0,0){\line(1,0){4}}
\put(0,2){\line(1,0){4}}
\put(0,4){\line(1,0){4}}
\put(0,0){\line(0,1){4}}
\put(2,0){\line(0,1){4}}
\put(4,0){\line(0,1){4}}

\put(2,2){\line(1,-1){2.1}}

}

\put(6,0){


\put(-1,4.3){?} \put(1.3,4.3){3} \put(4.3,4.3){1} 
\put(-1,2.3){1} \put(1.3,2.3){2} \put(4.3,2.3){3} 
\put(-1,-0.8){3} \put(1.3,-0.8){1} \put(4.3,-0.8){4} 

\put(0,4){\p} \put(1.8,3.8){$\circ$} \put(4,4){\p} 
\put(-0.2,1.8){$\circ$} \put(2,2){\p} \put(4,2){\p} 
\put(0,0){\p} \put(2,0){\p} \put(4,0){\p} 

\put(0,0){\line(1,0){4}}
\put(0,2){\line(1,0){4}}
\put(0,4){\line(1,0){4}}
\put(0,0){\line(0,1){4}}
\put(2,0){\line(0,1){4}}
\put(4,0){\line(0,1){4}}

\put(0,0){\line(1,1){2.1}}
\put(2,2){\line(1,1){2.1}}
\put(2,2){\line(1,-1){2.1}}

}

\put(15,0){


\put(-1,4.3){1} \put(1.3,4.3){3} \put(4.3,4.3){1} 
\put(-1,2.3){3} \put(1.3,1){2} \put(4.3,2.3){3} 
\put(-1,-0.8){?} \put(1.3,-0.8){1} \put(4.3,-0.8){4} 

\put(0,4){\p} \put(1.8,3.8){$\circ$} \put(4,4){\p} 
\put(-0.2,1.8){$\circ$} \put(2,2){\p} \put(4,2){\p} 
\put(0,0){\p} \put(2,0){\p} \put(4,0){\p} 

\put(0,0){\line(1,0){4}}
\put(0,2){\line(1,0){4}}
\put(0,4){\line(1,0){4}}
\put(0,0){\line(0,1){4}}
\put(2,0){\line(0,1){4}}
\put(4,0){\line(0,1){4}}

\put(2,2){\line(1,1){2.1}}
\put(0,4){\line(1,-1){2.1}}
\put(2,2){\line(1,-1){2.1}}

}

\end{picture}
\caption{Four subcases in Case 1 in Figure~\ref{conflict-situations}.} \label{case1-subcases}
\end{center}
\end{figure}
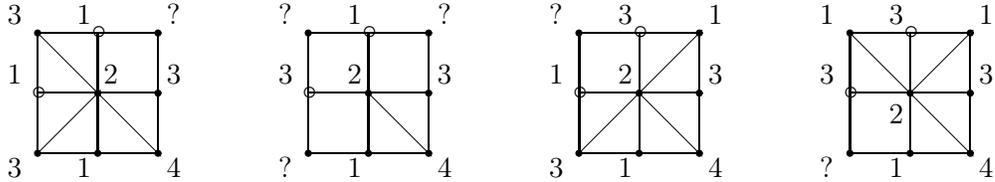

\noindent
{\bf Case 2: Triangulation $C_2$ in Figure~\ref{conflict-situations}}. In this case, $T$ must contain, as an induced subgraph, a $3\times 3$ grid graph triangulation with the bottom middle vertex being $v$. We have two possible subcases depending on the color of the vertex indicated by white circle in Figure~\ref{case2-subcases}, which allows us, in each case, to partially recover the triangulation of the $3\times 3$ grid graph involved, as well as some of vertices' colors. However, in each of the cases, there is a unique way to complete the triangulation, namely, by joining the vertices colored by 2 and 3. Indeed, if the center vertex colored by 1 is connected to the ? vertex, then the ? vertex must be colored by 4 contradicting the fact that $v$ was the first vertex colored by 4 in our colouring procedure. We see that in each case, the triangulation belongs to $S$.\\

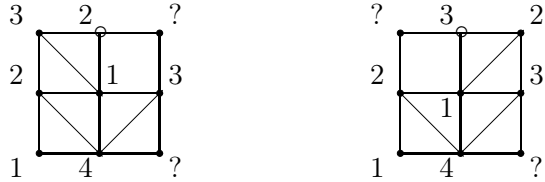
\begin{figure}[h]
\begin{center}
\begin{picture}(6,4.5)

\put(-5,0){


\put(-1,4.3){3} \put(1.3,4.3){2} \put(4.3,4.3){?} 
\put(-1,2.3){2} \put(2.2,2.3){1} \put(4.3,2.3){3} 
\put(-1,-0.8){1} \put(1.3,-0.8){4} \put(4.3,-0.8){?} 

\put(0,4){\p} \put(1.8,3.8){$\circ$} \put(4,4){\p} 
\put(0,2){\p} \put(2,2){\p} \put(4,2){\p} 
\put(0,0){\p} \put(2,0){\p} \put(4,0){\p} 

\put(0,0){\line(1,0){4}}
\put(0,2){\line(1,0){4}}
\put(0,4){\line(1,0){4}}
\put(0,0){\line(0,1){4}}
\put(2,0){\line(0,1){4}}
\put(4,0){\line(0,1){4}}

\put(0,2){\line(1,-1){2.1}}
\put(0,4){\line(1,-1){2.1}}
\put(2,0){\line(1,1){2.1}}

}

\put(7,0){


\put(-1,4.3){?} \put(1.3,4.3){3} \put(4.3,4.3){2} 
\put(-1,2.3){2} \put(1.3,1.2){1} \put(4.3,2.3){3} 
\put(-1,-0.8){1} \put(1.3,-0.8){4} \put(4.3,-0.8){?} 

\put(0,4){\p} \put(1.8,3.8){$\circ$} \put(4,4){\p} 
\put(0,2){\p} \put(2,2){\p} \put(4,2){\p} 
\put(0,0){\p} \put(2,0){\p} \put(4,0){\p} 

\put(0,0){\line(1,0){4}}
\put(0,2){\line(1,0){4}}
\put(0,4){\line(1,0){4}}
\put(0,0){\line(0,1){4}}
\put(2,0){\line(0,1){4}}
\put(4,0){\line(0,1){4}}

\put(0,2){\line(1,-1){2.1}}
\put(2,2){\line(1,1){2.1}}
\put(2,0){\line(1,1){2.1}}

}

\end{picture}
\caption{Two subcases in Case 2 in Figure~\ref{conflict-situations}.} \label{case2-subcases}
\end{center}
\end{figure}

\noindent
{\bf Case 3: Triangulation $C_3$ in Figure~\ref{conflict-situations}}. In this case, similarly to Case 2, $T$ must contain, as an induced subgraph, a $3\times 3$ grid graph triangulation with the bottom middle vertex being $v$. We have two possible subcases depending on the color of the vertex indicated by a white circle in Figure~\ref{case3-subcases}, which allows us in one case to recover partially the triangulation of the $3\times 3$ grid graph involved, while in the other case to do it completely obtaining as a result a triangulation in $S$. Also, we can recover some of vertices' colors. We now see that  in the partially recovered case, there is a unique way to complete the triangulation, namely, by joining the vertices colored by 1 and 3. Indeed, if the center vertex colored by 2 is connected to one of the ? vertices, then the ? vertex must be colored by 4 contradicting the fact that $v$ was the first vertex colored by 4 in our colouring procedure. We see that in any case, the triangulation belongs to~$S$. \end{proof}

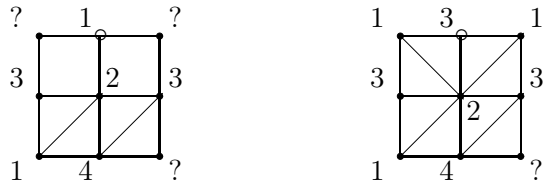
\begin{figure}[h]
\begin{center}
\begin{picture}(6,4.5)

\put(-5,0){


\put(-1,4.3){?} \put(1.3,4.3){1} \put(4.3,4.3){?} 
\put(-1,2.3){3} \put(2.2,2.3){2} \put(4.3,2.3){3} 
\put(-1,-0.8){1} \put(1.3,-0.8){4} \put(4.3,-0.8){?} 

\put(0,4){\p} \put(1.8,3.8){$\circ$} \put(4,4){\p} 
\put(0,2){\p} \put(2,2){\p} \put(4,2){\p} 
\put(0,0){\p} \put(2,0){\p} \put(4,0){\p} 

\put(0,0){\line(1,0){4}}
\put(0,2){\line(1,0){4}}
\put(0,4){\line(1,0){4}}
\put(0,0){\line(0,1){4}}
\put(2,0){\line(0,1){4}}
\put(4,0){\line(0,1){4}}

\put(0,0){\line(1,1){2.1}}
\put(2,0){\line(1,1){2.1}}

}

\put(7,0){


\put(-1,4.3){1} \put(1.3,4.3){3} \put(4.3,4.3){1} 
\put(-1,2.3){3} \put(2.2,1.2){2} \put(4.3,2.3){3} 
\put(-1,-0.8){1} \put(1.3,-0.8){4} \put(4.3,-0.8){?} 

\put(0,4){\p} \put(1.8,3.8){$\circ$} \put(4,4){\p} 
\put(0,2){\p} \put(2,2){\p} \put(4,2){\p} 
\put(0,0){\p} \put(2,0){\p} \put(4,0){\p} 

\put(0,0){\line(1,0){4}}
\put(0,2){\line(1,0){4}}
\put(0,4){\line(1,0){4}}
\put(0,0){\line(0,1){4}}
\put(2,0){\line(0,1){4}}
\put(4,0){\line(0,1){4}}

\put(0,0){\line(1,1){2.1}}
\put(2,2){\line(1,1){2.1}}
\put(2,0){\line(1,1){2.1}}
\put(0,4){\line(1,-1){2.1}}

}

\end{picture}
\caption{Two subcases in Case 3 in Figure~\ref{conflict-situations}.} \label{case3-subcases}
\end{center}
\end{figure}

By Lemma~\ref{lemma-grid} and Theorem~\ref{T1-T2} we have the truth of the following statement. 

\begin{theorem} A triangulation of a grid graph is word-representable if and only if it is $3$-colorable.\end{theorem}

\subsection{Triangulations of a convex polyomino}\label{tr-conv-polyomino}

Recall that $S$ is the set of eight graphs formed by rotations of $T_1$ and $T_2$ in Figure~\ref{non-repr-triang} by degrees multiple to $90\,^{\circ}$.

\begin{lemma}\label{lemma-conv-polyomino} A triangulation $T$ of a convex polyomino is $3$-colorable if and only if it does not contain a graph from $S$ as an induced subgraph.  \end{lemma}

\begin{proof} Assume that $T$ is not $3$-colorable and thus it can be colored in four colors. Our proof is an extension of the proof of Lemma~\ref{lemma-grid}. We use the same approach to color vertices of a triangulation $T$  of a convex polyomino as in the proof of  Lemma~\ref{lemma-grid} until we are forced 
to use color 4. We will show that either $T$ contains a graph from $S$ as an induced subgraph, or the vertices colored so far can be recolored to avoid usage of color 4; in the later case our arguments can be repeated until eventually it will be shown that $T$ contains a graph from $S$ (otherwise a contradiction would be obtained with $T$ being non-3-colorable). Once again, there are three possible situations that are shown in Figure~\ref{conflict-situations-polyomino}, where the areas of the convex polyomino labeled by $A$, $B$ and $C$ can possibly contain no other vertices than those shown in the figure colored by 1, 2 and 3. We assume that vertices on the boundary of two areas belong to both areas; in particular, the vertex colored by 2 in the leftmost picture in Figure~\ref{conflict-situations-polyomino} belongs to all three areas.  Finally, we call a vertex in an area {\em internal}, if it belongs only to a single area.\\

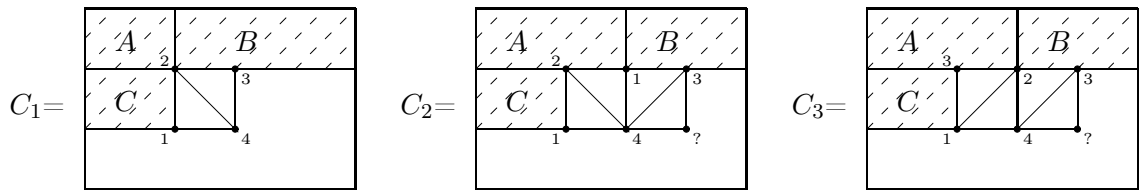
\begin{figure}[h]
\begin{center}
\begin{picture}(8,7)


\put(-12,0){

\put(-2.5,2.5){$C_1$=}

\put(2.5,1.5){{\tiny 1}} 
\put(2.5,4.1){{\tiny 2}} 
\put(5.2,1.5){{\tiny 4}} 
\put(5.2,3.5){{\tiny 3}} 

\put(3,2){\p} \put(5,2){\p} \put(3,4){\p} \put(5,4){\p} 

\put(0,0){\line(1,0){9}}
\put(0,6){\line(1,0){9}}
\put(0,2){\line(1,0){5}}
\put(3,4){\line(1,0){6}}
\put(3,2){\line(0,1){2}}
\put(5,2){\line(0,1){2}}
\put(3,4){\line(1,-1){2.1}}
\put(0,0){\line(0,1){6}}
\put(9,0){\line(0,1){6}}

\multiput(0,2)(0.5,0.5){8}{\line(1,1){0.2}}
\multiput(1,2)(0.5,0.5){8}{\line(1,1){0.2}}
\multiput(2,2)(0.5,0.5){2}{\line(1,1){0.2}}
\multiput(0,3)(0.5,0.5){6}{\line(1,1){0.2}}
\multiput(0,4)(0.5,0.5){4}{\line(1,1){0.2}}
\multiput(0,5)(0.5,0.5){2}{\line(1,1){0.2}}
\multiput(4,4)(0.5,0.5){4}{\line(1,1){0.2}}
\multiput(5,4)(0.5,0.5){4}{\line(1,1){0.2}}
\multiput(6,4)(0.5,0.5){4}{\line(1,1){0.2}}
\multiput(7,4)(0.5,0.5){4}{\line(1,1){0.2}}
\multiput(8,4)(0.5,0.5){2}{\line(1,1){0.2}}


\put(0,4){\line(1,0){3}}
\put(3,4){\line(0,1){2}}
\put(1,4.6){$A$}
\put(1,2.6){$C$}
\put(5,4.6){$B$}

}


\put(1,0){

\put(-2.5,2.5){$C_2$=}

\put(2.5,1.5){{\tiny 1}} 
\put(2.5,4.1){{\tiny 2}} 
\put(5.2,1.5){{\tiny 4}} 
\put(5.2,3.5){{\tiny 1}} 
\put(7.2,3.5){{\tiny 3}} 
\put(7.2,1.5){{\tiny ?}} 

\put(3,2){\p} \put(5,2){\p} \put(3,4){\p} \put(5,4){\p} \put(7,4){\p} \put(7,2){\p} 

\put(5,2){\line(1,1){2.1}}
\put(0,0){\line(1,0){9}}
\put(0,6){\line(1,0){9}}
\put(0,2){\line(1,0){7}}
\put(3,4){\line(1,0){6}}
\put(3,2){\line(0,1){2}}
\put(5,2){\line(0,1){2}}
\put(3,4){\line(1,-1){2.1}}
\put(0,0){\line(0,1){6}}
\put(9,0){\line(0,1){6}}
\put(7,2){\line(0,1){2}}

\multiput(0,2)(0.5,0.5){8}{\line(1,1){0.2}}
\multiput(1,2)(0.5,0.5){8}{\line(1,1){0.2}}
\multiput(2,2)(0.5,0.5){2}{\line(1,1){0.2}}
\multiput(0,3)(0.5,0.5){6}{\line(1,1){0.2}}
\multiput(0,4)(0.5,0.5){4}{\line(1,1){0.2}}
\multiput(0,5)(0.5,0.5){2}{\line(1,1){0.2}}
\multiput(4,4)(0.5,0.5){4}{\line(1,1){0.2}}
\multiput(5,4)(0.5,0.5){4}{\line(1,1){0.2}}
\multiput(6,4)(0.5,0.5){4}{\line(1,1){0.2}}
\multiput(7,4)(0.5,0.5){4}{\line(1,1){0.2}}
\multiput(8,4)(0.5,0.5){2}{\line(1,1){0.2}}


\put(0,4){\line(1,0){3}}
\put(5,4){\line(0,1){2}}
\put(1,4.6){$A$}
\put(1,2.6){$C$}
\put(6,4.6){$B$}

}


\put(14,0){

\put(-2.5,2.5){$C_3$=}

\put(2.5,1.5){{\tiny 1}} 
\put(2.5,4.1){{\tiny 3}} 
\put(5.2,1.5){{\tiny 4}} 
\put(5.2,3.5){{\tiny 2}} 
\put(7.2,3.5){{\tiny 3}} 
\put(7.2,1.5){{\tiny ?}} 

\put(3,2){\p} \put(5,2){\p} \put(3,4){\p} \put(5,4){\p} \put(7,4){\p} \put(7,2){\p} 

\put(5,2){\line(1,1){2.1}}
\put(0,0){\line(1,0){9}}
\put(0,6){\line(1,0){9}}
\put(0,2){\line(1,0){7}}
\put(3,4){\line(1,0){6}}
\put(3,2){\line(0,1){2}}
\put(5,2){\line(0,1){2}}
\put(3,2){\line(1,1){2.1}}
\put(0,0){\line(0,1){6}}
\put(9,0){\line(0,1){6}}
\put(7,2){\line(0,1){2}}

\multiput(0,2)(0.5,0.5){8}{\line(1,1){0.2}}
\multiput(1,2)(0.5,0.5){8}{\line(1,1){0.2}}
\multiput(2,2)(0.5,0.5){2}{\line(1,1){0.2}}
\multiput(0,3)(0.5,0.5){6}{\line(1,1){0.2}}
\multiput(0,4)(0.5,0.5){4}{\line(1,1){0.2}}
\multiput(0,5)(0.5,0.5){2}{\line(1,1){0.2}}
\multiput(4,4)(0.5,0.5){4}{\line(1,1){0.2}}
\multiput(5,4)(0.5,0.5){4}{\line(1,1){0.2}}
\multiput(6,4)(0.5,0.5){4}{\line(1,1){0.2}}
\multiput(7,4)(0.5,0.5){4}{\line(1,1){0.2}}
\multiput(8,4)(0.5,0.5){2}{\line(1,1){0.2}}


\put(0,4){\line(1,0){3}}
\put(5,4){\line(0,1){2}}
\put(1,4.6){$A$}
\put(1,2.6){$C$}
\put(6,4.6){$B$}

}

\end{picture}
\caption{Three possible cases of appearance of color 4 in coloring of $T$.} \label{conflict-situations-polyomino}
\end{center}
\end{figure}

\noindent
{\bf Case 1: Triangulation $C_1$ in Figure~\ref{conflict-situations-polyomino}}.  We consider two subcases here: \\

\noindent
{\bf Subcase 1.1:} $A$ has at least one internal vertex. If $B$ and $C$ each have at least one internal vertex,  then, taking into account that the polyomino is convex, $T$ has a $3\times 3$ grid graph with the bottom rightmost vertex colored by 4 as an induced subgraph, and exactly the same arguments as those in the proof of Case 1 in Lemma~\ref{lemma-grid} can be applied to see that $T$ contains a graph from $S$ as an induced subgraph. 
On the other hand, if $B$ (resp., $C$) does not have an internal vertex, then the vertex colored by $3$ (resp., $1$) in the picture could be recolored in $1$ (resp., $3$) so that there would be no need for color 4, and we would continue colouring $T$ until color $4$ needs to be used (we would then find ourselves considering again one of the three cases with more vertices already colored).\\

\noindent
{\bf Subcase 1.2:} $A$ does not have any internal vertices. We can recolor vertices of $C$ as follows: $1\rightarrow 3$, $2\rightarrow 2$ and $3\rightarrow 1$ (in particular, vertices colored by 2 keep the same color). Recoloring  does not affect coloring in $B$, that is, we still have a proper coloring of a part of $T$. But then we see that usage of 4 is unnecessary: that vertex can be recolored using color $1$, and we can continue coloring $T$ until color 4 needs to be used.\\

\noindent
{\bf Case 2: Triangulation $C_2$ in Figure~\ref{conflict-situations-polyomino}}.  Again, we consider two subcases here: \\

\noindent
{\bf Subcase 1.1:} $A$ has at least one internal vertex. If $B$ has an internal vertex then, taking into account convexity, we can use the argument in the proof of Case 2 in Lemma~\ref{lemma-grid} applied to the  $3\times 3$ grid graph with the bottom rightmost vertex marked by ? to obtain the desired. On the other hand, if $B$ has no internal nodes, then the bottom border of $B$ (containing at least  two vertices colored 1 and 3) can be recolored as  $1\rightarrow 1$, $2\rightarrow 3$ and $3\rightarrow 2$  keeping the property of being a proper coloring. The recoloring shows that the usage of color 4 was unnecessary --- color 3 can be used instead, and coloring $T$ can be continued until there is a need of usage of color $4$.\\

\noindent
{\bf Subcase 1.2:} $A$ does not  have an internal vertex. In this case, $C$ can be recolored as $1\rightarrow 1$, $2\rightarrow 3$ and $3\rightarrow 2$.  The recoloring shows that the usage of color 4 was unnecessary --- color 2 can be used instead, and coloring $T$ can be continued until there is a need for color $4$.\\

\noindent
{\bf Case 3: Triangulation $C_3$ in Figure~\ref{conflict-situations-polyomino}}.   Once again, we consider two subcases here: \\

\noindent
{\bf Subcase 1.1:} $A$ has at least one internal vertex. If $B$ has an internal vertex then, taking into account convexity, we can use the argument in the proof of Case 2 in Lemma~\ref{lemma-grid} applied to the  $3\times 3$ grid graph with the bottom rightmost vertex marked by ? to obtain the desired.   On the other hand, if $B$ has no internal nodes, then the bottom border of $B$ (containing at least  two vertices colored 2 and 3) can be recolored as  $1\rightarrow 3$, $2\rightarrow 2$ and $3\rightarrow 1$  keeping the property of being a proper coloring. The recoloring shows that the usage of color 4 was unnecessary --- color 3 can be used instead, and coloring $T$ can be continued until there is a need for usage of color $4$. \\

\noindent
{\bf Subcase 1.2:} $A$ does not  have an internal vertex. In this final case, $C$ can be recolored as $1\rightarrow 3$, $2\rightarrow 2$ and $3\rightarrow 1$.  The recoloring shows that the usage of color 4 was unnecessary --- color 1 can be used instead, and coloring $T$ can be continued until there is a need for color $4$.\\

\vspace{-0.8cm} \end{proof}

Our main result in this paper, Theorem~\ref{main-thm}, now follows from Lemma~\ref{lemma-conv-polyomino} and Theorem~\ref{T1-T2}.

\subsection{Triangulations of an arbitrary polyomino}\label{triangArbitrary}

Theorem~\ref{main-thm} is not true for triangulations of an arbitrary polyomino. Indeed, consider the triangulation $T$ of the polyomino on 7 squares in Figure~\ref{4-color-repr-polyomino} (where the center square does not belong to the polyomino). It is easy to see that $T$ is not 3-colorable, e.g. by letting the top leftmost vertex be colored by 1 and the vertex horizontally next to it be colored by 2.  On the other hand, $T$ accepts a semi-transitive orientation shown to the right in Figure~\ref{4-color-repr-polyomino}. This orientation is obtained by using a coloring of $T$ in four colors described above and orienting edges following the rules:
$$1\rightarrow 2, 1\rightarrow 3, 1\rightarrow 4, 2\rightarrow 3, 2\rightarrow 4, \mbox{\ and\ }4\rightarrow 3.$$ To see that the orientation is semi-transitive, one can observe that the only possible shortcuts must have a directed path $1\rightarrow 2\rightarrow 4\rightarrow 3$ and the edge $1\rightarrow 3$; however, there are only three directed paths $1\rightarrow 2\rightarrow 4\rightarrow 3$ in the orientation, and in each case beginning and ending of such a path are not connected by an edge $1\rightarrow 3$.

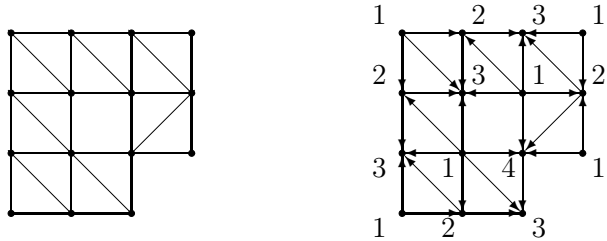
\begin{figure}[h]
\begin{center}
\begin{picture}(7,6.5)

\put(-7,0){



\put(0,6){\p} \put(2,6){\p} \put(4,6){\p} \put(6,6){\p} 
\put(0,4){\p} \put(2,4){\p} \put(4,4){\p} \put(6,4){\p} 
\put(0,2){\p} \put(2,2){\p} \put(4,2){\p} \put(6,2){\p} 
\put(0,0){\p} \put(2,0){\p} \put(4,0){\p} 

\put(0,0){\line(1,0){4}}
\put(0,2){\line(1,0){6}}
\put(0,4){\line(1,0){6}}
\put(0,6){\line(1,0){6}}
\put(0,0){\line(0,1){6}}
\put(2,0){\line(0,1){6}}
\put(4,0){\line(0,1){6}}
\put(6,2){\line(0,1){4}}

\put(2,0){\line(-1,1){2.1}}
\put(2,2){\line(-1,1){2.1}}
\put(2,4){\line(-1,1){2.1}}
\put(4,0){\line(-1,1){2.1}}
\put(4,4){\line(-1,1){2.1}}
\put(4,2){\line(1,1){2.1}}
\put(6,4){\line(-1,1){2.1}}


}


\put(6,0){

\put(-1,6.3){1} \put(2.3,6.3){2} \put(4.3,6.3){3} \put(6.3,6.3){1} 
\put(-1,4.3){2} \put(2.3,4.3){3} \put(4.3,4.3){1} \put(6.3,4.3){2} 
\put(-1,1.2){3} \put(1.3,1.2){1} \put(3.3,1.2){4} \put(6.3,1.2){1} 
\put(-1,-0.8){1} \put(1.3,-0.8){2} \put(4.3,-0.8){3} 

\put(0,6){\p} \put(2,6){\p} \put(4,6){\p} \put(6,6){\p} 
\put(0,4){\p} \put(2,4){\p} \put(4,4){\p} \put(6,4){\p} 
\put(0,2){\p} \put(2,2){\p} \put(4,2){\p} \put(6,2){\p} 
\put(0,0){\p} \put(2,0){\p} \put(4,0){\p} 

\put(2,0){\vector(-1,1){1.9}}
\put(0,0){\vector(1,0){1.9}}
\put(0,0){\vector(0,1){1.9}}
\put(2,2){\vector(0,-1){1.9}}
\put(2,0){\vector(1,0){1.9}}
\put(2,2){\vector(1,-1){1.9}}
\put(2,2){\vector(1,0){1.9}}
\put(4,2){\vector(0,-1){1.9}}
\put(6,2){\vector(-1,0){1.9}}
\put(6,2){\vector(0,1){1.9}}
\put(2,2){\vector(-1,0){1.9}}
\put(0,4){\vector(0,-1){1.9}}
\put(0,4){\vector(1,0){1.9}}
\put(2,2){\vector(-1,1){1.9}}
\put(2,2){\vector(0,1){1.9}}
\put(4,4){\vector(-1,0){1.9}}
\put(4,4){\vector(0,-1){1.9}}
\put(4,4){\vector(1,0){1.9}}
\put(6,4){\vector(-1,-1){1.9}}
\put(0,6){\vector(0,-1){1.9}}
\put(0,6){\vector(1,-1){1.9}}
\put(0,6){\vector(1,0){1.9}}
\put(2,6){\vector(0,-1){1.9}}
\put(2,6){\vector(1,0){1.9}}
\put(4,4){\vector(-1,1){1.9}}
\put(4,4){\vector(0,1){1.9}}
\put(6,4){\vector(-1,1){1.9}}
\put(6,6){\vector(0,-1){1.9}}
\put(6,6){\vector(-1,0){1.9}}


}

\end{picture}
\caption{A non-3-colorable triangulation of a polyomino, and its semi-transitive orientation.} \label{4-color-repr-polyomino}
\end{center}
\end{figure}

\section{Replacing 4-cycles in a polyomino by $K_4$}\label{sec4}

In this section we discuss an operation on an {\em arbitrary} polyomino that {\em always} results in a word-representable graph, as opposed to triangulations considered above. Namely, we consider replacement of each 4-cycle in a polyomino with the complete graph $K_4$; see Figure~\ref{K4-replacement-ex-1} for a respective example. 

\begin{figure}[h]
\begin{center}
\begin{picture}(7,4.5)

\put(-6,0){


\put(0,4){\p} \put(2,4){\p} \put(4,4){\p} \put(6,4){\p} 
\put(0,2){\p} \put(2,2){\p} \put(4,2){\p} \put(6,2){\p} 
\put(0,0){\p} \put(2,0){\p} \put(4,0){\p} 

\put(0,0){\line(1,0){4}}
\put(0,2){\line(1,0){6}}
\put(0,4){\line(1,0){6}}
\put(0,0){\line(0,1){4}}
\put(2,0){\line(0,1){4}}
\put(4,0){\line(0,1){4}}
\put(6,2){\line(0,1){2}}

}

\put(4,0){


\put(0,4){\p} \put(2,4){\p} \put(4,4){\p} \put(6,4){\p} 
\put(0,2){\p} \put(2,2){\p} \put(4,2){\p} \put(6,2){\p} 
\put(0,0){\p} \put(2,0){\p} \put(4,0){\p} 

\put(0,0){\line(1,0){4}}
\put(0,2){\line(1,0){6}}
\put(0,4){\line(1,0){6}}
\put(0,0){\line(0,1){4}}
\put(2,0){\line(0,1){4}}
\put(4,0){\line(0,1){4}}
\put(6,2){\line(0,1){2}}

\put(0,0){\line(1,1){4}}
\put(2,0){\line(1,1){4}}
\put(0,2){\line(1,1){2}}
\put(0,4){\line(1,-1){4}}
\put(2,4){\line(1,-1){2}}
\put(4,4){\line(1,-1){2}}
\put(0,2){\line(1,-1){2}}

}

\end{picture}
\caption{An example of replacement of 4-cycles in a polyomino by $K_4$.} \label{K4-replacement-ex-1}
\end{center}
\end{figure}
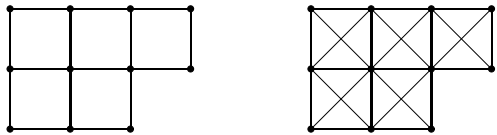

\begin{theorem}\label{K4-is-word-repr} Replacing each $4$-cycle in a polyomino $\mathcal{P}$ by $K_4$ gives a word-representable graph $\mathcal{P}_{K_4}$.\end{theorem}

\begin{proof} We begin with providing a semi-transitive orientation of the graph $G$ obtained from a grid graph by replacing each 4-cycle by $K_4$, and then we discuss the case of an arbitrary polyomino.

We call a vertex in an oriented copy of $G$ a {\em horizontal sink} (resp., {\em horizontal source}) if there are  no horizontal edges coming out of (resp., coming in to) the vertex. 

A semi-transitive orientation of $G$ can now be described as follows. Make the top row of $G$, as well as any odd row from the top, be a sequence of alternating horizontal sources and sinks (from left to right), while all other rows be a sequence of alternating horizontal sinks and sources as shown in Figure~\ref{K4-replacement-ex}. Moreover, we orient all other edges of $G$ downwards (see Figure~\ref{K4-replacement-ex}). Clearly, the orientation is acyclic. Furthermore, it is easy to see that the orientation of $G$ is semi-transitive. Indeed, starting from a vertex $v$ and looking at paths of length 3 or more ending in vertex $u$, we see by inspection that there are only two possibilities: 
\begin{itemize}
\item $v$ and $u$ are not connected by an edge (most common situation) thus giving no chance for a shortcut;
\item there is an edge from $v$ to $u$. In this case, the directed path must be of length 3 and it must cover three external edges  of a $K_4$ oriented transitively. Thus, we do not have a shortcut in this case either.
\end{itemize}

\begin{figure}[h]
\begin{center}
\begin{picture}(9,6.5)

\put(0,0){


\put(0,6){\p} \put(2,6){\p} \put(4,6){\p} \put(6,6){\p} \put(8,6){\p} \put(10,6){\p}
\put(0,4){\p} \put(2,4){\p} \put(4,4){\p} \put(6,4){\p} \put(8,4){\p} \put(10,4){\p}
\put(0,2){\p} \put(2,2){\p} \put(4,2){\p} \put(6,2){\p} \put(8,2){\p} \put(10,2){\p}
\put(0,0){\p} \put(2,0){\p} \put(4,0){\p} \put(6,0){\p} \put(8,0){\p} \put(10,0){\p}

\put(0,6){\vector(1,0){1.9}} \put(4,6){\vector(-1,0){1.9}}\put(4,6){\vector(1,0){1.9}}  \put(8,6){\vector(-1,0){1.9}} \put(8,6){\vector(1,0){1.9}} 
\put(2,4){\vector(-1,0){1.9}} \put(2,4){\vector(1,0){1.9}}\put(6,4){\vector(-1,0){1.9}}  \put(6,4){\vector(1,0){1.9}} \put(10,4){\vector(-1,0){1.9}} 
\put(0,2){\vector(1,0){1.9}} \put(4,2){\vector(-1,0){1.9}}\put(4,2){\vector(1,0){1.9}}  \put(8,2){\vector(-1,0){1.9}} \put(8,2){\vector(1,0){1.9}} 
\put(2,0){\vector(-1,0){1.9}} \put(2,0){\vector(1,0){1.9}}\put(6,0){\vector(-1,0){1.9}}  \put(6,0){\vector(1,0){1.9}} \put(10,0){\vector(-1,0){1.9}} 

\put(0,6){\vector(0,-1){1.9}} \put(2,6){\vector(0,-1){1.9}} \put(4,6){\vector(0,-1){1.9}} \put(6,6){\vector(0,-1){1.9}} \put(8,6){\vector(0,-1){1.9}} \put(10,6){\vector(0,-1){1.9}} 
\put(0,4){\vector(0,-1){1.9}} \put(2,4){\vector(0,-1){1.9}} \put(4,4){\vector(0,-1){1.9}} \put(6,4){\vector(0,-1){1.9}} \put(8,4){\vector(0,-1){1.9}} \put(10,4){\vector(0,-1){1.9}} 
\put(0,2){\vector(0,-1){1.9}} \put(2,2){\vector(0,-1){1.9}} \put(4,2){\vector(0,-1){1.9}} \put(6,2){\vector(0,-1){1.9}} \put(8,2){\vector(0,-1){1.9}} \put(10,2){\vector(0,-1){1.9}} 

\put(0,6){\vector(1,-1){1.9}} \put(2,6){\vector(1,-1){1.9}} \put(4,6){\vector(1,-1){1.9}} \put(6,6){\vector(1,-1){1.9}}  \put(8,6){\vector(1,-1){1.9}}  
\put(0,4){\vector(1,-1){1.9}} \put(2,4){\vector(1,-1){1.9}} \put(4,4){\vector(1,-1){1.9}} \put(6,4){\vector(1,-1){1.9}}  \put(8,4){\vector(1,-1){1.9}}  
\put(0,2){\vector(1,-1){1.9}} \put(2,2){\vector(1,-1){1.9}} \put(4,2){\vector(1,-1){1.9}} \put(6,2){\vector(1,-1){1.9}}  \put(8,2){\vector(1,-1){1.9}}  

\put(2,6){\vector(-1,-1){1.9}} \put(4,6){\vector(-1,-1){1.9}} \put(6,6){\vector(-1,-1){1.9}} \put(8,6){\vector(-1,-1){1.9}}  \put(10,6){\vector(-1,-1){1.9}}  
\put(2,4){\vector(-1,-1){1.9}} \put(4,4){\vector(-1,-1){1.9}} \put(6,4){\vector(-1,-1){1.9}} \put(8,4){\vector(-1,-1){1.9}}  \put(10,4){\vector(-1,-1){1.9}}  
\put(2,2){\vector(-1,-1){1.9}} \put(4,2){\vector(-1,-1){1.9}} \put(6,2){\vector(-1,-1){1.9}} \put(8,2){\vector(-1,-1){1.9}}  \put(10,2){\vector(-1,-1){1.9}}  

}

\end{picture}
\caption{A semi-transitive orientation of $G$.} \label{K4-replacement-ex}
\end{center}
\end{figure}
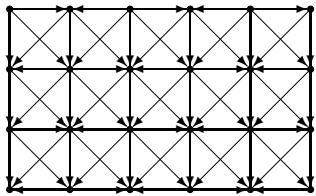

It is now clear how to semi-transitively orient $\mathcal{P}_{K_4}$ for any polyomino $\mathcal{P}$ rather than just for a grid graph. Indeed, $\mathcal{P}_{K_4}$ can be extended to a grid graph $G$ by adding missing $K_4$'s, then we can orient $G$ as above, and finally remove the $K_4$'s that were just added to obtain a semi-transitive orientation of   $\mathcal{P}_{K_4}$ (it is easy to see that removing $K_4$'s from the oriented $G$ cannot introduce any shortcuts). 
\end{proof}

We note that in our orientation of $G$ in the proof of Theorem~\ref{K4-is-word-repr} (see Figure~\ref{K4-replacement-ex}), horizontal sources in different rows are never on top of each other. A similar observation applies to horizontal sinks. If we were to eliminate this condition, thus making odd columns consist of horizontal sources and even columns of horizontal sinks keeping the vertical edges be oriented in the same way, we would obtain an orientation having shortcuts (e.g., see the induced subgraph formed by the first two vertices in the top row and the second and third vertex in the second row). 

\begin{figure}[h]
\begin{center}
\begin{picture}(6,8.5)

\put(-10,0){

\put(1,8){\p} \put(1,6){\p} \put(1,4){\p} \put(1,2){\p} 
\put(0,7.7){1} \put(0,5.7){2} \put(0,3.7){3} \put(0,1.7){4} 
\put(1,8){\line(0,-1){6}}
\put(0.9,0.2){$\vdots$}

\put(2,5){$\longmapsto$}

\put(5,8){\p} \put(7,8){\p} \put(6,6){\p} \put(6,4){\p} \put(6,2){\p} 
\put(4,7.7){1} \put(7.4,7.7){$1'$} \put(5,5.7){2} \put(5,3.7){3} \put(5,1.7){4} 
\put(6,6){\line(0,-1){4}}\put(5,8){\line(1,0){2}}\put(5,8){\line(1,-2){1}}\put(7,8){\line(-1,-2){1}}
\put(5.9,0.2){$\vdots$}

\put(8,5){$\longmapsto$}

\put(11,8){\p} \put(13,8){\p} \put(11,6){\p}  \put(13,6){\p}  \put(12,4){\p} \put(12,2){\p} 
\put(10,7.7){1} \put(13.4,7.7){$1'$} \put(10,5.7){2} \put(13.4,5.7){$2'$} \put(11,3.7){3} \put(11,1.7){4} 
\put(12,4){\line(0,-1){2}}\put(11,8){\line(1,0){2}} \put(11,6){\line(1,0){2}}  \put(11,6){\line(0,1){2}}  \put(11,6){\line(1,1){2}}  
\put(13,6){\line(0,1){2}} \put(13,6){\line(-1,1){2}} \put(11,6){\line(1,-2){1}}\put(13,6){\line(-1,-2){1}}
\put(11.9,0.2){$\vdots$}

\put(14.5,5){$\longmapsto$}
\put(17,5){$\cdots$}
\put(19,5){$\longmapsto$}

\put(23,8){\p} \put(25,8){\p} \put(23,8){\line(1,0){2}} \put(23,8){\line(1,-1){2}} 
\put(23,6){\p} \put(25,6){\p} \put(23,6){\line(1,0){2}} \put(23,6){\line(1,-1){2}} \put(23,6){\line(1,1){2}} 
\put(23,4){\p} \put(25,4){\p} \put(23,4){\line(1,0){2}} \put(23,4){\line(1,-1){2}} \put(23,4){\line(1,1){2}} 
\put(23,2){\p} \put(25,2){\p} \put(23,2){\line(1,0){2}}                                       \put(23,2){\line(1,1){2}} 

\put(22,7.7){1} \put(25.4,7.7){$1'$} 
\put(22,5.7){2} \put(25.4,5.7){$2'$} 
\put(22,3.7){3} \put(25.4,3.7){$3'$} 
\put(22,1.7){4} \put(25.4,1.7){$4'$} 

\put(23,8){\line(0,-1){6}} 
\put(25,8){\line(0,-1){6}} 

\put(23.9,0.2){$\vdots$}

}

\end{picture}
\caption{Generating $\mathcal{P}_{K_4}$ for an $n\times 2$ grid graph from a path graph on $n$ vertices through replacing vertices by the module $K_2$.} \label{K4-n-2-grid}
\end{center}
\end{figure}
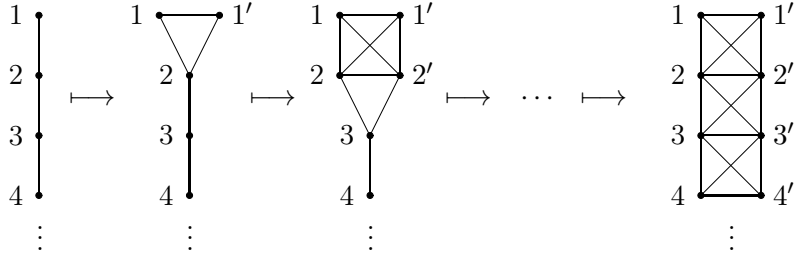

Another remark is that when $\mathcal{P}$ is an $n\times 2$ (equivalently, $2\times n$) grid graph in Theorem~\ref{K4-is-word-repr} (the respective graph $\mathcal{P}_{K_4}$ is presented schematically to the right in Figure~\ref{K4-n-2-grid}), then we can prove this particular case of the theorem by other means. Indeed,  $\mathcal{P}_{K_4}$ can be obtained in this case by substitution the vertices of a path graph (presented to the left in Figure~\ref{K4-n-2-grid}) by {\em modules} $K_2$, as shown schematically in Figure~\ref{K4-n-2-grid}. It is a known fact, see e.g.~\cite[Section 5]{K} that replacing a vertex in a graph by a module turns a word-representable graph into a word-representable graph showing that $\mathcal{P}_{K_4}$ in this case is word-presentable because the path graph on $n$ vertices we started with is such a graph (a particular representation of the graph is $12132435465\ldots$). Alternatively, we can come up directly with a word representing the graph $\mathcal{P}_{K_4}$, e.g. representing it by the word $11'22'11'33'22'44'33'55'\ldots$. However, none of these approaches work, at least that easily, for larger grid graphs with 4-cycles replaced by $K_4$, so it is essential to employ semi-transitive orientations here to prove our results.

\section{Final remarks}\label{final-remarks-sec}  

When we began our studies of triangulations of grid graphs and convex polyominoes from word-representability point of view, we were not hoping for such an elegant result as Theorem~\ref{main-thm}, which deals with an arbitrary triangulation of an arbitrary convex polyomino. As we demonstrated in Subsection~\ref{triangArbitrary}, unfortunately, Theorem~\ref{main-thm} does not hold for an arbitrary polyomino. We leave this as an open question to give a classification of situations when a triangulation of a given (arbitrary) polyomino is word-representable.

\end{document}